\documentclass[final,10pt,a4paper]{my_elsarticle}

\usepackage{amsfonts,amsmath,amsthm}
\usepackage{graphicx}
\usepackage{url}

\newtheorem{theorem}{Theorem}[section]
\newtheorem{corollary}{Corollary}[section]
\newtheorem{example}{Example}[section]

\newtheorem{lemma}[theorem]{Lemma}
\newtheorem{proposition}{Proposition}[section]

\newtheorem{definition}[theorem]{Definition}
\newtheorem{remark}{Remark}[section]

\begin{document}

\begin{frontmatter}

\title{Discrete-Time Fractional Variational Problems}

\author[labelNuno]{Nuno R. O. Bastos}
\ead{nbastos@estv.ipv.pt}

\author[labelRui]{Rui A. C. Ferreira}
\ead{ruiacferreira@ua.pt}

\author[labelDelfim]{Delfim F. M. Torres}
\ead{delfim@ua.pt}

\address[labelNuno]{Department of Mathematics, ESTGV\\
Polytechnic Institute of Viseu\\
3504-510 Viseu, Portugal}

\address[labelRui]{Faculty of Engineering and Natural Sciences\\
Lusophone University of Humanities and Technologies\\
1749-024 Lisbon, Portugal}

\address[labelDelfim]{Department of Mathematics\\
University of Aveiro\\
3810-193 Aveiro, Portugal}

\begin{abstract}
We introduce a discrete-time fractional calculus of variations on
the time scale $h\mathbb{Z}$, $h > 0$. First and second order necessary
optimality conditions are established. Examples illustrating the use
of the new Euler-Lagrange and Legendre type conditions are given.
They show that solutions to the considered fractional problems
become the classical discrete-time solutions when the
fractional order of the discrete-derivatives are integer values,
and that they converge to the fractional continuous-time solutions
when $h$ tends to zero. Our Legendre type condition
is useful to eliminate false candidates identified
via the Euler-Lagrange fractional equation.
\end{abstract}


\end{frontmatter}


\section{Introduction}
\label{int}

The \emph{Fractional Calculus} (calculus with derivatives of arbitrary order)
is an important research field in several different areas such as physics
(including classical and quantum mechanics as well as thermodynamics),
chemistry, biology, economics, and control theory
\cite{agr3,B:08,Miller1,Ortigueira,TenreiroMachado}.
It has its origin more than 300 years ago when L'Hopital asked Leibniz what
should be the meaning of a derivative of non-integer order. After
that episode several more famous mathematicians contributed to the
development of Fractional Calculus: Abel, Fourier, Liouville,
Riemann, Riesz, just to mention a few names \cite{book:Kilbas,Samko}.
In the last decades, considerable research has been done in fractional calculus.
This is particularly true in the area of the calculus of variations,
which is being subject to intense investigations during the last few years
\cite{B:08,Ozlem,R:N:H:M:B:07,R:T:M:B:07}.
Applications include fractional variational principles in
mechanics and physics, quantization, control theory, and description of conservative,
nonconservative, and constrained systems \cite{B:08,B:M:08,B:M:R:08,R:T:M:B:07}.
Roughly speaking, the classical calculus of variations and optimal control is extended
by substituting the usual derivatives of integer order
by different kinds of fractional (non-integer) derivatives.
It is important to note that the passage from the integer/classical differential calculus
to the fractional one is not unique because we have at our disposal
different notions of fractional derivatives.
This is, as argued in \cite{B:08,R:N:H:M:B:07},
an interesting and advantage feature of the area.
Most part of investigations in the fractional variational calculus are
based on the replacement of the classical derivatives by fractional derivatives
in the sense of Riemann--Liouville, Caputo, Riesz, and Jumarie
\cite{agr0,MyID:182,B:08,MyID:149}.
Independently of the chosen fractional derivatives,
one obtains, when the fractional order of differentiation tends
to an integer order, the usual problems and results of the calculus
of variations. Although the fractional Euler--Lagrange equations
are obtained in a similar manner as in the standard variational calculus \cite{R:N:H:M:B:07},
some classical results are extremely difficult to be proved in a fractional context.
This explains, for example, why a fractional Legendre type condition
is absent from the literature of fractional variational calculus.
In this work we give a first result in this direction
(\textrm{cf.} Theorem~\ref{thm1}).

Despite its importance in applications, less is known for
discrete-time fractional systems \cite{R:N:H:M:B:07}.
In \cite{Miller} Miller and Ross define a fractional sum of order
$\nu>0$ \emph{via} the solution of a linear difference equation.
They introduce it as (see \S\ref{sec0} for the notations used here)
\begin{equation}
\label{naosei8}
\Delta^{-\nu}f(t)=\frac{1}{\Gamma(\nu)}\sum_{s=a}^{t-\nu}(t-\sigma(s))^{(\nu-1)}f(s).
\end{equation}
Definition \eqref{naosei8} is analogous
to the Riemann-Liouville fractional integral
$$
_a\mathbf{D}_x^{-\nu}f(x)=\frac{1}{\Gamma(\nu)}\int_{a}^{x}(x-s)^{\nu-1}f(s)ds
$$
of order $\nu>0$, which can be obtained \emph{via} the solution of a linear
differential equation \cite{Miller,Miller1}. Basic properties
of the operator $\Delta^{-\nu}$ in (\ref{naosei8}) were obtained in \cite{Miller}.
More recently, Atici and Eloe introduced the
fractional difference of order $\alpha > 0$ by
$\Delta^\alpha f(t)=\Delta^m(\Delta^{\alpha - m}f(t))$,
where $m$ is the integer part of $\alpha$,
and developed some of its properties that allow
to obtain solutions of certain
fractional difference equations \cite{Atici0,Atici}.

The fractional differential calculus has been widely developed in the
past few decades due mainly to its demonstrated applications in
various fields of science and engineering \cite{book:Kilbas,Miller1,Podlubny}.
The study of necessary optimality conditions for
fractional problems of the calculus of variations
and optimal control is a fairly recent issue attracting an increasing attention
-- see \cite{agr0,agr2,RicDel,El-Nabulsi1,El-Nabulsi2,gastao:delfim,gasta1,M:Baleanu}
and references therein -- but available results address only the continuous-time case.
It is well known that discrete analogues of differential equations can be very useful in
applications \cite{B:J:06,J:B:07,book:DCV} and that fractional Euler-Lagrange differential
equations are extremely difficult to solve, being necessary to discretize them \cite{agr2,Ozlem}.
Therefore, it is pertinent to develop a fractional discrete-time theory of the
calculus of variations for the time scale $(h\mathbb{Z})_a$, $h > 0$
(\textrm{cf.} definitions in Section~\ref{sec0}).
Computer simulations show that this time scale is particularly interesting
because when $h$ tends to zero one recovers previous fractional continuous-time results.

Our objective is two-fold. On one hand we proceed to develop the
theory of \emph{fractional difference calculus}, namely, we
introduce the concept of left and right fractional sum/difference
(\textrm{cf.} Definition~\ref{def0}). On the other hand, we believe
that the present work will potentiate research not only in the
fractional calculus of variations but also in solving fractional
difference equations, specifically, fractional equations in which
left and right fractional differences appear.
Because the theory of fractional difference calculus is still in its
infancy \cite{Atici0,Atici,Miller}, the paper is self contained.
In \S\ref{sec0} we introduce notations, we give necessary definitions,
and prove some preliminary results needed in the sequel.
Main results of the paper appear in \S\ref{sec1}:
we prove a fractional formula of $h$-summation by parts
(Theorem~\ref{teor1}), and necessary optimality
conditions of first and second order
(Theorems~\ref{thm0} and \ref{thm1}, respectively)
for the proposed $h$-fractional problem of the calculus of
variations \eqref{naosei7}. Section~\ref{sec2} gives some illustrative examples,
and we end the paper with \S\ref{sec:conc} of conclusions and future perspectives.

The results of the paper are formulated using standard notations
of the theory of time scales \cite{livro:2001,J:B:M:08,malina}.
It remains an interesting open question how to generalize
the present results to an arbitrary time scale $\mathbb{T}$.
This is a difficult and challenging problem since our proofs
deeply rely on the fact that in $\mathbb{T} = (h\mathbb{Z})_a$
the graininess function is a constant.


\section{Preliminaries}
\label{sec0}

We begin by recalling the main definitions and properties of time
scales (\textrm{cf.}~\cite{CD:Bohner:2004,livro:2001} and references
therein). A nonempty closed subset of $\mathbb{R}$ is called a \emph{time
scale} and is denoted by $\mathbb{T}$. The \emph{forward jump operator}
$\sigma:\mathbb{T}\rightarrow\mathbb{T}$ is defined by
$\sigma(t)=\inf{\{s\in\mathbb{T}:s>t}\}$ for all $t\in\mathbb{T}$,
while the \emph{backward jump operator}
$\rho:\mathbb{T}\rightarrow\mathbb{T}$ is defined by
$\rho(t)=\sup{\{s\in\mathbb{T}:s<t\}}$ for all
$t\in\mathbb{T}$, with $\inf\emptyset=\sup\mathbb{T}$
(\textrm{i.e.}, $\sigma(M)=M$ if $\mathbb{T}$ has a maximum $M$) and
$\sup\emptyset=\inf\mathbb{T}$ (\textrm{i.e.}, $\rho(m)=m$ if
$\mathbb{T}$ has a minimum $m$).
A point $t\in\mathbb{T}$ is called \emph{right-dense},
\emph{right-scattered}, \emph{left-dense}, or \emph{left-scattered},
if $\sigma(t)=t$, $\sigma(t)>t$, $\rho(t)=t$, or $\rho(t)<t$,
respectively. Throughout the text we let $\mathbb{T}=[a,b]\cap\tilde{\mathbb{T}}$
with $a<b$ and $\tilde{\mathbb{T}}$ a time scale. We define
$\mathbb{T}^\kappa=\mathbb{T}\backslash(\rho(b),b]$,
$\mathbb{T}^{\kappa^2}=\left(\mathbb{T}^\kappa\right)^\kappa$
and more generally $\mathbb{T}^{\kappa^n}
=\left(\mathbb{T}^{\kappa^{n-1}}\right)^\kappa$, for
$n\in\mathbb{N}$. The following standard notation is used for
$\sigma$ (and $\rho$): $\sigma^0(t) = t$, $\sigma^n(t) = (\sigma
\circ \sigma^{n-1})(t)$, $n \in \mathbb{N}$.
The \emph{graininess function} $\mu:\mathbb{T}\rightarrow[0,\infty)$
is defined by $\mu(t)=\sigma(t)-t$ for all $t\in\mathbb{T}$.

A function $f:\mathbb{T}\rightarrow\mathbb{R}$ is said to be
\emph{delta differentiable} at $t\in\mathbb{T}^\kappa$ if there is a
number $f^{\Delta}(t)$ such that for all $\varepsilon>0$ there
exists a neighborhood $U$ of $t$ (\textrm{i.e.},
$U=(t-\delta,t+\delta)\cap\mathbb{T}$ for some $\delta>0$) such that
$$|f(\sigma(t))-f(s)-f^{\Delta}(t)(\sigma(t)-s)|
\leq\varepsilon|\sigma(t)-s|,\mbox{ for all $s\in U$}.$$ We call
$f^{\Delta}(t)$ the \emph{delta derivative} of $f$ at $t$.
The $r^{th}-$\emph{delta derivative}
($r\in\mathbb{N}$) of $f$ is defined to be the function
$f^{\Delta^r}:\mathbb{T}^{\kappa^r}\rightarrow\mathbb{R}$, provided
$f^{\Delta^{r-1}}$ is delta differentiable on $\mathbb{T}^{\kappa^{r-1}}$.
For delta differentiable $f$ and $g$ and for an arbitrary
time scale $\mathbb{T}$ the next formulas hold:
$f^\sigma(t) = f(t)+\mu(t)f^\Delta(t)$ and
\begin{equation}
\label{deltaderpartes2}
(fg)^\Delta(t) = f^\Delta(t)g^\sigma(t)+f(t)g^\Delta(t)
=f^\Delta(t)g(t)+f^\sigma(t)g^\Delta(t),
\end{equation}
where we abbreviate $f\circ\sigma$ by $f^\sigma$.
A function $f:\mathbb{T}\rightarrow\mathbb{R}$ is called
\emph{rd-continuous} if it is continuous at right-dense points and
if its left-sided limit exists at left-dense points. The
set of all rd-continuous functions is denoted by $C_{\textrm{rd}}$
and the set of all delta differentiable functions
with rd-continuous derivative by $C_{\textrm{rd}}^1$.
It is known that rd-continuous functions possess an
\emph{antiderivative}, \textrm{i.e.}, there exists a function
$F \in C_{\textrm{rd}}^1$ with $F^\Delta=f$. The delta \emph{integral} is then defined by
$\int_{a}^{b}f(t)\Delta t=F(b)-F(a)$. It satisfies the equality
$\int_t^{\sigma(t)}f(\tau)\Delta\tau=\mu(t)f(t)$. We make use
of the following properties of the delta integral:
\begin{lemma}(\textrm{cf.} \cite[Theorem~1.77]{livro:2001})
\label{integracao:partes}
If $a,b\in\mathbb{T}$ and
$f,g\in$C$_{\textrm{rd}}$, then
\begin{enumerate}

 \item$\int_{a}^{b}f(\sigma(t))g^{\Delta}(t)\Delta t
 =\left.(fg)(t)\right|_{t=a}^{t=b}-\int_{a}^{b}f^{\Delta}(t)g(t)\Delta t$;

\item $\int_{a}^{b}f(t)g^{\Delta}(t)\Delta t
=\left.(fg)(t)\right|_{t=a}^{t=b}-\int_{a}^{b}f^{\Delta}(t)g(\sigma(t))\Delta t$.
\end{enumerate}
\end{lemma}

One way to approach the Riemann-Liouville fractional calculus
is through the theory of linear differential equations \cite{Podlubny}.
Miller and Ross \cite{Miller} use an analogous methodology to introduce
fractional discrete operators for the case
$\mathbb{T}=\mathbb{Z}_a=\{a,a+1,a+2,\ldots\}, a\in\mathbb{R}$.
Here we go a step further: we use the theory of time scales
in order to introduce fractional discrete operators
to the more general case
$\mathbb{T}=(h\mathbb{Z})_a=\{a,a+h,a+2h,\ldots\}$,
$a\in\mathbb{R}$, $h>0$.

For $n\in \mathbb{N}_0$ and rd-continuous functions
$p_i:\mathbb{T}\rightarrow \mathbb{R}$, $1\leq i\leq n$,
let us consider the $n$th order linear dynamic equation
\begin{equation}
\label{linearDiffequa}
Ly=0\, , \quad \text{ where } Ly=y^{\Delta^n}+\sum_{i=1}^n
p_iy^{\Delta^{n-i}} \, .
\end{equation}
A function $y:\mathbb{T}\rightarrow \mathbb{R}$ is said to be a
solution of equation (\ref{linearDiffequa}) on $\mathbb{T}$ provided
$y$ is $n$ times delta differentiable on $\mathbb{T}^{\kappa^n}$ and
satisfies $Ly(t)=0$ for all $t\in\mathbb{T}^{\kappa^n}$.

\begin{lemma}\cite[p.~239]{livro:2001}
\label{8:88 Bohner}
If $z=\left(z_1, \ldots,z_n\right) : \mathbb{T}
\rightarrow \mathbb{R}^n$ satisfies for all $t\in
\mathbb{T}^\kappa$
\begin{equation}\label{5.86}
z^\Delta=A(t)z(t),\qquad\mbox{where}\qquad A=\left(
                                               \begin{array}{ccccc}
                                                 0 & 1 & 0 & \ldots & 0 \\
                                                 \vdots & 0 & 1 & \ddots & \vdots \\
                                                 \vdots &  & \ddots & \ddots & 0 \\
                                                 0 & \ldots & \ldots & 0 & 1 \\
                                                 -p_n & \ldots & \ldots & -p_2 & -p_1 \\
                                               \end{array}
                                             \right)
\end{equation}
then $y=z_1$ is a solution of equation \eqref{linearDiffequa}.
Conversely, if $y$ solves \eqref{linearDiffequa} on $\mathbb{T}$,
then $z=\left(y, y^\Delta,\ldots, y^{\Delta^{n-1}}\right) : \mathbb{T}\rightarrow \mathbb{R}$ satisfies \eqref{5.86} for
all $t\in\mathbb{T}^{\kappa^n}$
\end{lemma}

\begin{definition}\cite[p.~239]{livro:2001}
We say that equation (\ref{linearDiffequa}) is \emph{regressive}
provided $I + \mu(t) A(t)$ is invertible for all $t \in \mathbb{T}^\kappa$,
where $A$ is the matrix in \eqref{5.86}.
\end{definition}

\begin{definition}\cite[p.~250]{livro:2001}
We define the Cauchy function $y:\mathbb{T} \times \mathbb{T}^{\kappa^n}\rightarrow \mathbb{R}$
for the linear dynamic equation~(\ref{linearDiffequa}) to be,
for each fixed $s\in\mathbb{T}^{\kappa^n}$, the solution of the initial value problem
\begin{equation}
\label{IVP}
Ly=0,\quad y^{\Delta^i}\left((\sigma(s),s\right)=0,\quad 0\leq i
\leq n-2,\quad y^{\Delta^{n-1}}\left((\sigma(s),s\right)=1\, .
\end{equation}
\end{definition}

\begin{theorem}\cite[p.~251]{livro:2001}\label{eqsol}
Suppose $\{y_1,\ldots,y_n\}$ is a fundamental system of the
regressive equation~(\ref{linearDiffequa}). Let $f\in C_{rd}$. Then
the solution of the initial value problem
$$
Ly=f(t),\quad y^{\Delta^i}(t_0)=0,\quad 0\leq i\leq n-1 \, ,
$$
is given by
$y(t)=\int_{t_0}^t y(t,s)f(s)\Delta s$,
where $y(t,s)$ is the Cauchy function for ~(\ref{linearDiffequa}).
\end{theorem}

It is known that $y(t,s):=H_{n-1}(t,\sigma(s))$
is the Cauchy function for $y^{\Delta^n}=0$,
where $H_{n-1}$ is a time scale generalized polynomial
\cite[Example~5.115]{livro:2001}.
The generalized polynomials $H_{k}$ are the functions
$H_k:\mathbb{T}^2\rightarrow \mathbb{R}$, $k\in \mathbb{N}_0$,
defined recursively as follows:
\begin{equation*}
H_0(t,s)\equiv 1 \, , \quad
H_{k+1}(t,s)=\int_s^t H_k(\tau,s)\Delta\tau \, , \quad
k = 1, 2, \ldots
\end{equation*}
for all $s,t\in \mathbb{T}$.
If we let $H_k^\Delta(t,s)$ denote, for each fixed $s$, the derivative
of $H_k(t,s)$ with respect to $t$, then (\textrm{cf.} \cite[p.~38]{livro:2001})
$$
H_k^\Delta(t,s)=H_{k-1}(t,s)\quad \text{for }
k\in \mathbb{N}, \ t\in \mathbb{T}^\kappa \, .
$$

From now on we restrict ourselves to the time scale
$\mathbb{T}=(h\mathbb{Z})_a$, $h > 0$, for which
the graininess function is the constant $h$.
Our main goal is to propose and develop a discrete-time
fractional variational theory in $\mathbb{T}=(h\mathbb{Z})_a$.
We borrow the notations from the recent calculus of variations on time scales
\cite{CD:Bohner:2004,RD,J:B:M:08}. How to generalize our results
to an arbitrary time scale $\mathbb{T}$,
with the graininess function $\mu$ depending on time, is not clear
and remains a challenging question.

Let $a\in\mathbb{R}$ and $h>0$,
$(h\mathbb{Z})_a=\{a,a+h,a+2h,\ldots\}$, and
$b=a+kh$ for some $k\in\mathbb{N}$. We have $\sigma(t)=t+h$,
$\rho(t)=t-h$, $\mu(t) \equiv h$, and we will frequently write
$f^\sigma(t)=f(\sigma(t))$. We put
$\mathbb{T}=[a,b]\cap (h\mathbb{Z})_a$, so that
$\mathbb{T}^\kappa=[a,\rho(b)]\cap (h\mathbb{Z})_a$ and
$\mathbb{T}^{\kappa^2}=[a,\rho^2(b)] \cap (h\mathbb{Z})_a$.
The delta derivative coincides in this case with
the forward $h$-difference:
$\displaystyle{f^{\Delta}(t)=\frac{f^\sigma(t)-f(t)}{\mu(t)}}$. If
$h=1$, then we have the usual discrete forward difference $\Delta f(t)$.
The delta integral gives the $h$-sum (or $h$-integral) of $f$:
$\displaystyle{\int_a^b
f(t)\Delta t=\sum_{k=\frac{a}{h}}^{\frac{b}{h}-1}f(kh)h}$. If we
have a function $f$ of two variables, $f(t,s)$, its partial
forward $h$-differences will be denoted by $\Delta_{t,h}$ and
$\Delta_{s,h}$, respectively. We will make use of the standard conventions
$\sum_{t=c}^{c-1}f(t)=0$, $c\in\mathbb{Z}$,
and $\prod_{i=0}^{-1}f(i)=1$.
Often, \emph{left fractional delta integration} (resp.,
\emph{right fractional delta integration}) of order $\nu>0$
is denoted by $_a\Delta_t^{-\nu}f(t)$ (resp. $_t\Delta_b^{-\nu}f(t)$).
Here, similarly as in Ross \textit{et. al.} \cite{Ross}, where the authors omit
the subscript $t$ on the operator (the operator itself cannot depend
on $t$), we write $_a\Delta_h^{-\nu}f(t)$ (resp. $_h\Delta_b^{-\nu}f(t)$).

Before giving an explicit formula for the generalized polynomials $H_{k}$ on $h\mathbb{Z}$
we introduce the following definition:
\begin{definition}
For arbitrary $x,y\in\mathbb{R}$ the $h$-factorial function is defined by
\begin{equation*}
x_h^{(y)}:=h^y\frac{\Gamma(\frac{x}{h}+1)}{\Gamma(\frac{x}{h}+1-y)}\, ,
\end{equation*}
where $\Gamma$ is the well-known Euler gamma function, and
we use the convention that division at a pole yields zero.
\end{definition}

\begin{remark}
For $h = 1$, and in accordance with
the previous literature \eqref{naosei8},
we write $x^{(y)}$ to denote $x_h^{(y)}$.
\end{remark}

\begin{proposition}
\label{prop:d}
For the time-scale $\mathbb{T}=(h\mathbb{Z})_a$ one has
\begin{equation}
\label{hn}
H_{k}(t,s):=\frac{(t-s)_h^{(k)}}{k!}\quad\mbox{for all}\quad s,t\in \mathbb{T}
\text{ and } k\in \mathbb{N}_0 \, .
\end{equation}
\end{proposition}

To prove \eqref{hn} we use the following technical lemma.
Throughout the text the basic property
$\Gamma(x+1)=x\Gamma(x)$ of the gamma function will be frequently used.

\begin{lemma}
\label{lem:tl}
Let $s \in \mathbb{T}$. Then,
for all $t \in \mathbb{T}^\kappa$ one has
\begin{equation*}
\Delta_{t,h} \left\{\frac{(t-s)_h^{(k+1)}}{(k+1)!}\right\}
= \frac{(t-s)_h^{(k)}}{k!} \, .
\end{equation*}
\end{lemma}
\begin{proof}
The equality follows by direct computations:
\begin{equation*}
\begin{split}
\Delta_{t,h} &\left\{\frac{(t-s)_h^{(k+1)}}{(k+1)!}\right\}
=\frac{1}{h}\left\{\frac{(\sigma(t)-s)_h^{(k+1)}}{(k+1)!}-\frac{(t-s)_h^{(k+1)}}{(k+1)!}\right\}\\
&=\frac{h^{k+1}}{h(k+1)!}\left\{\frac{\Gamma((t+h-s)/h+1)}{\Gamma((t+h-s)/h+1-(k+1))}-\frac{\Gamma((t-s)/h+1)}{\Gamma((t-s)/h+1-(k+1))}\right\}\\
&=\frac{h^k}{(k+1)!}\left\{\frac{((t-s)/h+1)\Gamma((t-s)/h+1)}{((t-s)/h-k)\Gamma((t-s)/h-k)}-\frac{\Gamma((t-s)/h+1)}{\Gamma((t-s)/h-k)}\right\}\\
&=\frac{h^k}{k!}\left\{\frac{\Gamma((t-s)/h+1)}{\Gamma((t-s)/h+1-k)}\right\} =\frac{(t-s)_h^{(k)}}{k!} \, .
\end{split}
\end{equation*}
\end{proof}
\begin{proof}(of Proposition~\ref{prop:d})
We proceed by mathematical induction. For $k=0$
$$
H_0(t,s)=\frac{1}{0!}h^0\frac{\Gamma(\frac{t-s}{h}+1)}{\Gamma(\frac{t-s}{h}+1-0)}
=\frac{\Gamma(\frac{t-s}{h}+1)}{\Gamma(\frac{t-s}{h}+1)}=1 \, .
$$
Assume that (\ref{hn}) holds for $k$ replaced by $m$. Then by Lemma~\ref{lem:tl}
\begin{eqnarray*}
H_{m+1}(t,s) &=& \int_s^t H_m(\tau,s)\Delta\tau
= \int_s^t \frac{(\tau-s)_h^{(m)}}{m!} \Delta\tau
= \frac{(t-s)_h^{(m+1)}}{(m+1)!},
\end{eqnarray*}
which is (\ref{hn}) with $k$ replaced by $m+1$.
\end{proof}

Let $y_1(t),\ldots,y_n(t)$ be $n$ linearly independent solutions of
the linear homogeneous dynamic equation $y^{\Delta^n}=0$.
From Theorem~\ref{eqsol} we know that the solution of
\eqref{IVP} (with $L=\Delta^n$ and $t_0=a$) is
\begin{equation*}
y(t) = \Delta^{-n} f(t)=\int_a^t
\frac{(t-\sigma(s))_h^{(n-1)}}{\Gamma(n)}f(s)\Delta s\\
=\frac{1}{\Gamma(n)}\sum_{k=a/h}^{t/h-1}
(t-\sigma(kh))_h^{(n-1)} f(kh) h \, .
\end{equation*}
Since $y^{\Delta_i}(a)=0$, $i = 0,\ldots,n-1$, then we can write that
\begin{equation}
\label{eq:derDh:int}
\begin{split}
\Delta^{-n} f(t)
&= \frac{1}{\Gamma(n)}\sum_{k=a/h}^{t/h-n}
(t-\sigma(kh))_h^{(n-1)} f(kh) h \\
&= \frac{1}{\Gamma(n)}\int_a^{\sigma(t-nh)}(t-\sigma(s))_h^{(n-1)}f(s) \Delta s \, .
\end{split}
\end{equation}
Note that function $t \rightarrow (\Delta^{-n} f)(t)$ is defined for
$t=a+n h \mbox{ mod}(h)$ while function $t \rightarrow f(t)$ is defined for $t=a \mbox{
mod}(h)$. Extending \eqref{eq:derDh:int} to any positive real value $\nu$,
and having as an analogy the continuous left and right
fractional derivatives \cite{Miller1},
we define the left fractional $h$-sum and the right fractional $h$-sum as follows.
We denote by $\mathcal{F}_\mathbb{T}$ the set of all real valued functions
defined on a given time scale $\mathbb{T}$.

\begin{definition}
\label{def0}
Let $a\in\mathbb{R}$, $h>0$,
$b=a+kh$ with $k\in\mathbb{N}$, and put
$\mathbb{T}=[a,b]\cap(h\mathbb{Z})_a$.
Consider $f\in\mathcal{F}_\mathbb{T}$. The left and right fractional $h$-sum of order $\nu>0$
are, respectively, the operators $_a\Delta_h^{-\nu} : \mathcal{F}_\mathbb{T} \rightarrow \mathcal{F}_{\tilde{\mathbb{T}}_\nu^+}$
and $_h\Delta_b^{-\nu} : \mathcal{F}_\mathbb{T}
\rightarrow \mathcal{F}_{\tilde{\mathbb{T}}_\nu^-}$,
$\tilde{\mathbb{T}}_\nu^\pm = \{t \pm \nu h : t \in \mathbb{T}\}$, defined by
\begin{equation*}
\begin{split}
_a\Delta_h^{-\nu}f(t) &= \frac{1}{\Gamma(\nu)}\int_{a}^{\sigma(t-\nu h)}(t-\sigma(s))_h^{(\nu-1)}f(s)\Delta s
=\frac{1}{\Gamma(\nu)}\sum_{k=\frac{a}{h}}^{\frac{t}{h}-\nu}(t-\sigma(kh))_h^{(\nu-1)}f(kh)h\\
_h\Delta_b^{-\nu}f(t) &= \frac{1}{\Gamma(\nu)}\int_{t+\nu h}^{\sigma(b)}(s-\sigma(t))_h^{(\nu-1)}f(s)\Delta s
=\frac{1}{\Gamma(\nu)}\sum_{k=\frac{t}{h}+\nu}^{\frac{b}{h}}(kh-\sigma(t))_h^{(\nu-1)}f(kh)h.
\end{split}
\end{equation*}
\end{definition}

\begin{remark}
In Definition~\ref{def0} we are using summations with limits that are reals.
For example, the summation that appears in the definition
of operator $_a\Delta_h^{-\nu}$ has the following meaning:
$$
\sum_{k = \frac{a}{h}}^{\frac{t}{h} - \nu} G(k) =
G(a/h) + G(a/h+1) + G(a/h+2) + \cdots + G(t/h - \nu),
$$
where $t \in \{ a + \nu h, a + h + \nu h , a + 2 h + \nu h ,  \ldots, \underbrace{a+kh}_b + \nu h\}$
with $k\in\mathbb{N}$.
\end{remark}

\begin{lemma}
Let $\nu>0$ be an arbitrary positive real number. For any $t \in \mathbb{T}$ we have:
(i) $\lim_{\nu\rightarrow 0}{_a}\Delta_h^{-\nu}f(t+\nu h)=f(t)$;
(ii) $\lim_{\nu\rightarrow 0}{_h}\Delta_b^{-\nu}f(t-\nu h)=f(t)$.
\end{lemma}
\begin{proof}
Since
\begin{align*}
{_a}\Delta_h^{-\nu}f(t+\nu h)&=\frac{1}{\Gamma(\nu)}\int_{a}^{\sigma(t)}(t+\nu
h-\sigma(s))_h^{(\nu-1)}f(s)\Delta s\\
&=\frac{1}{\Gamma(\nu)}\sum_{k=\frac{a}{h}}^{\frac{t}{h}}(t+\nu h-\sigma(kh))_h^{(\nu-1)}f(kh)h\\
&=h^{\nu}f(t)+\frac{\nu}{\Gamma(\nu+1)}\sum_{k=\frac{a}{h}}^{\frac{\rho(t)}{h}}(t+\nu h-\sigma(kh))_h^{(\nu-1)}f(kh)h\, ,
\end{align*}
it follows that $\lim_{\nu\rightarrow 0}{_a}\Delta_h^{-\nu}f(t+\nu h)=f(t)$.
The proof of (ii) is similar.
\end{proof}

For any $t\in\mathbb{T}$ and for any $\nu\geq 0$
we define $_a\Delta_h^{0}f(t) := {_h}\Delta_b^{0}f(t) := f(t)$
and write
\begin{equation}
\label{seila1}
\begin{gathered}
{_a}\Delta_h^{-\nu}f(t+\nu h) = h^\nu f(t)
+\frac{\nu}{\Gamma(\nu+1)}\int_{a}^{t}(t+\nu
h-\sigma(s))_h^{(\nu-1)}f(s)\Delta s\, , \\
{_h}\Delta_b^{-\nu}f(t)=h^\nu f(t-\nu h)
+ \frac{\nu}{\Gamma(\nu+1)}\int_{\sigma(t)}^{\sigma(b)}(s+\nu
h-\sigma(t))_h^{(\nu-1)}f(s)\Delta s \, .
\end{gathered}
\end{equation}

\begin{theorem}\label{thm2}
Let $f\in\mathcal{F}_\mathbb{T}$ and $\nu\geq0$.
For all $t\in\mathbb{T}^\kappa$ we have
\begin{equation}
\label{naosei1}
{_a}\Delta_{h}^{-\nu}
f^{\Delta}(t+\nu h)=(_a\Delta_h^{-\nu}f(t+\nu h))^{\Delta}
-\frac{\nu}{\Gamma(\nu + 1)}(t+\nu h-a)_h^{(\nu-1)}f(a) \, .
\end{equation}
\end{theorem}
To prove Theorem~\ref{thm2} we make use of a technical lemma:
\begin{lemma}
\label{lemma:tl}
Let $t\in\mathbb{T}^\kappa$. The following
equality holds for all $s\in\mathbb{T}^\kappa$:
\begin{multline}
\label{proddiff}
\Delta_{s,h}\left((t+\nu h-s)_h^{(\nu-1)}f(s))\right)\\
=(t+\nu h-\sigma(s))_h^{(\nu-1)}f^{\Delta}(s) -(v-1)(t+\nu h-\sigma(s))_h^{(\nu-2)}f(s) \, .
\end{multline}
\end{lemma}
\begin{proof}
Direct calculations give the intended result:
\begin{equation*}
\begin{split}
\Delta&_{s,h} \left((t+\nu h-s)_h^{(\nu-1)}f(s)\right)\\
&=\Delta_{s,h}\left((t+\nu h-s)_h^{(\nu-1)}\right)f(s)+\left(t+\nu h
-\sigma(s)\right)_h^{(\nu-1)}f^{\Delta}(s)\\
&=\frac{f(s)}{h}\left[h^{\nu-1}\frac{\Gamma\left(\frac{t+\nu
h-\sigma(s)}{h}+1\right)}{\Gamma\left(\frac{t+\nu
h-\sigma(s)}{h}+1-(\nu-1)\right)}-h^{\nu-1}\frac{\Gamma\left(\frac{t+\nu
h-s}{h}+1\right)}{\Gamma\left(\frac{t+\nu
h-s}{h}+1-(\nu-1)\right)}\right]\\
&\qquad +\left(t+\nu h -
\sigma(s)\right)_h^{(\nu-1)}f^{\Delta}(s)\\
&=f(s)\left[h^{\nu-2}\left[\frac{\Gamma(\frac{t+\nu
h-s}{h})}{\Gamma(\frac{t-s}{h}+1)}-\frac{\Gamma(\frac{t+\nu
h-s}{h}+1)}{\Gamma(\frac{t-s}{h}+2)}\right]\right]+\left(t+\nu h -
\sigma(s)\right)_h^{(\nu-1)}f^{\Delta}(s)\\
&=f(s)h^{\nu-2}\frac{\Gamma(\frac{t+\nu
h-s-h}{h}+1)}{\Gamma(\frac{t-s+\nu h-h}{h}+1-(\nu-2))}(-(\nu-1))+
\left(t+\nu h - \sigma(s)\right)_h^{(\nu-1)}f^{\Delta}(s)\\
&=-(\nu-1)(t+\nu h -\sigma(s))_h^{(\nu-2)}f(s)+\left(t+\nu h
- \sigma(s)\right)_h^{(\nu-1)}f^{\Delta}(s) \, ,
\end{split}
\end{equation*}
where the first equality follows directly from \eqref{deltaderpartes2}.
\end{proof}

\begin{remark}
Given an arbitrary $t\in\mathbb{T}^\kappa$
it is easy to prove, in a similar way as in the proof of
Lemma~\ref{lemma:tl}, the following equality analogous to \eqref{proddiff}:
for all $s\in\mathbb{T}^\kappa$
\begin{multline}
\label{eq:semlhante}
\Delta_{s,h}\left((s+\nu h-\sigma(t))_h^{(\nu-1)}f(s))\right)\\
=(\nu-1)(s+\nu h-\sigma(t))_h^{(\nu-2)}f^\sigma(s) + (s+\nu h-\sigma(t))_h^{(\nu-1)}f^{\Delta}(s) \, .
\end{multline}
\end{remark}
\begin{proof}(of Theorem~\ref{thm2})
From Lemma~\ref{lemma:tl} we obtain that
\begin{equation}
\label{naosei}
\begin{split}
{_a}\Delta_{h}^{-\nu} & f^{\Delta}(t+\nu h)
= h^\nu f^\Delta(t)+\frac{\nu}{\Gamma(\nu+1)}\int_{a}^{t}(t+\nu
h-\sigma(s))_h^{(\nu-1)}f^{\Delta}(s)\Delta s\\
&=h^\nu f^\Delta(t)+\frac{\nu}{\Gamma(\nu+1)}\left[(t+\nu
h-s)_h^{(\nu-1)}f(s)\right]_{s=a}^{s=t}\\
&\qquad +\frac{\nu}{\Gamma(\nu+1)}\int_{a}^{\sigma(t)}(\nu-1)(t+\nu
h-\sigma(s))_h^{(\nu-2)}
f(s)\Delta s\\
&=-\frac{\nu(t+\nu h-a)_h^{(\nu-1)}}{\Gamma(\nu+1)}f(a)
+h^{\nu}f^\Delta(t)+\nu h^{\nu-1}f(t)\\
&\qquad +\frac{\nu}{\Gamma(\nu+1)}\int_{a}^{t}(\nu-1)(t+\nu
h-\sigma(s))_h^{(\nu-2)} f(s)\Delta s.
\end{split}
\end{equation}
We now show that $(_a\Delta_h^{-\nu}f(t+\nu h))^{\Delta}$ equals \eqref{naosei}:
\begin{equation*}
\begin{split}
(_a\Delta_h^{-\nu} & f(t+\nu h))^\Delta
= \frac{1}{h}\left[h^\nu f(\sigma(t))+\frac{\nu}{\Gamma(\nu+1)}\int_{a}^{\sigma(t)}(\sigma(t)+\nu
h-\sigma(s))_h^{(\nu-1)}
f(s)\Delta s\right.\\
&\qquad \left.-h^\nu f(t)-\frac{\nu}{\Gamma(\nu+1)}\int_{a}^{t}(t+\nu h-\sigma(s))_h^{(\nu-1)} f(s)\Delta s\right]\\
&=h^\nu f^\Delta(t)+\frac{\nu}{h\Gamma(\nu+1)}\left[\int_{a}^{t}(\sigma(t)+\nu
h-\sigma(s))_h^{(\nu-1)}
f(s)\Delta s\right.\\
&\qquad\left.-\int_{a}^{t}(t+\nu h-\sigma(s))_h^{(\nu-1)}
f(s)\Delta s\right]+h^{\nu-1}\nu f(t)\\
&=h^\nu f^\Delta(t)+\frac{\nu}{\Gamma(\nu+1)}\int_{a}^{t}\Delta_{t,h}\left((t+\nu
h -\sigma(s))_h^{(\nu-1)}
\right)f(s)\Delta s+h^{\nu-1}\nu f(t)\\
&=h^\nu f^\Delta(t)+\frac{\nu}{\Gamma(\nu+1)}\int_{a}^{t}(\nu-1)(t+\nu
h-\sigma(s))_h^{(\nu-2)} f(s)\Delta s+\nu h^{\nu-1}f(t) \, .
\end{split}
\end{equation*}
\end{proof}

Follows the counterpart of Theorem~\ref{thm2}
for the right fractional $h$-sum:
\begin{theorem}
\label{thm3}
Let $f\in\mathcal{F}_\mathbb{T}$ and $\nu\geq 0$.
For all $t\in\mathbb{T}^\kappa$ we have
\begin{equation}
\label{naosei12}
{_h}\Delta_{\rho(b)}^{-\nu}
f^{\Delta}(t-\nu h)=\frac{\nu}{\Gamma(\nu+1)}(b+\nu
h-\sigma(t))_h^{(\nu-1)}f(b)+(_h\Delta_b^{-\nu}f(t-\nu h))^{\Delta} \, .
\end{equation}
\end{theorem}
\begin{proof}
From \eqref{eq:semlhante} we obtain from integration by parts
(item 2 of Lemma~\ref{integracao:partes}) that
\begin{equation}
\label{naosei99}
\begin{split}
{_h}\Delta_{\rho(b)}^{-\nu} & f^{\Delta}(t-\nu h)
=\frac{\nu(b+\nu h-\sigma(t))_h^{(\nu-1)}}{\Gamma(\nu+1)}f(b)
+ h^\nu f^\Delta(t) -\nu h^{\nu-1}f(\sigma(t))\\
&\qquad -\frac{\nu}{\Gamma(\nu+1)}\int_{\sigma(t)}^{b}(\nu-1)(s+\nu
h-\sigma(t))_h^{(\nu-2)} f^\sigma(s)\Delta s.
\end{split}
\end{equation}
We show that $(_h\Delta_b^{-\nu}f(t-\nu h))^{\Delta}$ equals \eqref{naosei99}:
\begin{equation*}
\begin{split}
(_h&\Delta_b^{-\nu} f(t-\nu h))^{\Delta}\\
&=h^{\nu}f^\Delta(t)+\frac{\nu}{h\Gamma(\nu+1)}\left[\int_{\sigma^2(t)}^{\sigma(b)}(s+\nu
h-\sigma^2(t)))_h^{(\nu-1)}
f(s)\Delta s\right.\\
&\qquad \left.-\int_{\sigma^2(t)}^{\sigma(b)}(s+\nu h-\sigma(t))_h^{(\nu-1)}
f(s)\Delta s\right]-\nu h^{\nu-1} f(\sigma(t))\\
&=h^{\nu}f^\Delta(t)+\frac{\nu}{\Gamma(\nu+1)}\int_{\sigma^2(t)}^{\sigma(b)}\Delta_{t,h}\left((s+\nu
h-\sigma(t))_h^{(\nu-1)}\right)
f(s)\Delta s-\nu h^{\nu-1} f(\sigma(t))\\
&=h^{\nu}f^\Delta(t)-\frac{\nu}{\Gamma(\nu+1)}\int_{\sigma^2(t)}^{\sigma(b)}(\nu-1)(s+\nu
h-\sigma^2(t))_h^{(\nu-2)}
f(s)\Delta s-\nu h^{\nu-1} f(\sigma(t))\\
&=h^{\nu}f^\Delta(t)-\frac{\nu}{\Gamma(\nu+1)}\int_{\sigma(t)}^{b}(\nu-1)(s+\nu
h-\sigma(t))_h^{(\nu-2)} f(s)\Delta s-\nu h^{\nu-1} f(\sigma(t)).
\end{split}
\end{equation*}
\end{proof}

\begin{definition}
\label{def1}
Let $0<\alpha\leq 1$ and set $\gamma := 1-\alpha$. The \emph{left
fractional difference} $_a\Delta_h^\alpha f(t)$ and the
\emph{right fractional difference} $_h\Delta_b^\alpha f(t)$
of order $\alpha$ of a function $f\in\mathcal{F}_\mathbb{T}$ are defined as
\begin{equation*}
_a\Delta_h^\alpha f(t) := (_a\Delta_h^{-\gamma}f(t+\gamma h))^{\Delta}\ \text{ and } \
_h\Delta_b^\alpha f(t):=-(_h\Delta_b^{-\gamma}f(t-\gamma h))^{\Delta}
\end{equation*}
for all $t\in\mathbb{T}^\kappa$.
\end{definition}


\section{Main Results}
\label{sec1}

Our aim is to introduce the $h$-fractional problem of the calculus of
variations and to prove corresponding necessary optimality
conditions. In order to obtain an Euler-Lagrange
type equation (\textrm{cf.} Theorem~\ref{thm0}) we first prove a
fractional formula of $h$-summation by parts.


\subsection{Fractional $h$-summation by parts}

A big challenge was to discover a
fractional $h$-summation by parts formula within the time scale setting.
Indeed, there is no clue of what such a formula should be.
We found it eventually, making use of the following lemma.

\begin{lemma}
\label{lem1}
Let $f$ and $k$ be two functions defined on $\mathbb{T}^\kappa$
and $\mathbb{T}^{\kappa^2}$, respectively,
and $g$ a function defined on
$\mathbb{T}^\kappa\times\mathbb{T}^{\kappa^2}$.
The following equality holds:
\begin{equation*}
\int_{a}^{b}f(t)\left[\int_{a}^{t}g(t,s)k(s)\Delta s\right]\Delta
t=\int_{a}^{\rho(b)}k(t)\left[\int_{\sigma(t)}^{b}g(s,t)f(s)\Delta
s\right]\Delta t \, .
\end{equation*}
\end{lemma}
\begin{proof}
Consider the matrices
$R = \left[ f(a+h), f(a+2h), \cdots, f(b-h) \right]$,
\begin{equation*}
C_1 = \left[
\begin{array}{c}
g(a+h,a)k(a) \\
g(a+2h,a)k(a)+g(a+2h,a+h)k(a+h) \\
\vdots \\
g(b-h,a)k(a)+g(b-h,a+h)k(a+h)+\cdots+ g(b-h,b-2h)k(b-2h)
\end{array}
\right]
\end{equation*}
\begin{gather*}
C_2 = \left[
\begin{array}{c}
g(a+h,a) \\
g(a+2h,a) \\
\vdots \\
g(b-h,a)  \end{array} \right], \ \  C_3 = \left[
\begin{array}{c}
0 \\
g(a+2h,a+h) \\
\vdots \\
g(b-h,a+h)  \end{array} \right], \ \ C_4 = \left[
\begin{array}{c}
0 \\
0 \\
\vdots \\
g(b-h,b-2h)
\end{array}
\right] .
\end{gather*}
Direct calculations show that
\begin{equation*}
\begin{split}
\int_{a}^{b}&f(t)\left[\int_{a}^{t}g(t,s)k(s)\Delta s\right]\Delta t
=h^2\sum_{i=a/h}^{b/h-1} f(ih)\sum_{j=a/h}^{i-1}g(ih,jh)k(jh) = h^2 R \cdot C_1\\
&=h^2 R \cdot \left[k(a) C_2 + k(a+h)C_3 +\cdots +k(b-2h) C_4 \right]\\
&=h^2\left[k(a)\sum_{j=a/h+1}^{b/h-1}g(jh,a)f(jh)+k(a+h)\sum_{j=a/h+2}^{b/h-1}g(jh,a+h)f(jh)\right.\\
&\left.\qquad \qquad +\cdots+k(b-2h)\sum_{j=b/h-1}^{b/h-1}g(jh,b-2h)f(jh)\right]\\
&=\sum_{i=a/h}^{b/h-2}k(ih)h\sum_{j=\sigma(ih)/h}^{b/h-1}g(jh,ih)f(jh) h
=\int_a^{\rho(b)}k(t)\left[\int_{\sigma(t)}^b g(s,t)f(s)\Delta s\right]\Delta t.
\end{split}
\end{equation*}
\end{proof}

\begin{theorem}[fractional $h$-summation by parts]\label{teor1}
Let $f$ and $g$ be real valued functions defined on $\mathbb{T}^\kappa$
and $\mathbb{T}$, respectively. Fix $0<\alpha\leq 1$ and put
$\gamma := 1-\alpha$. Then,
\begin{multline}
\label{delf:sumPart}
\int_{a}^{b}f(t)_a\Delta_h^\alpha g(t)\Delta t=h^\gamma
f(\rho(b))g(b)-h^\gamma
f(a)g(a)+\int_{a}^{\rho(b)}{_h\Delta_{\rho(b)}^\alpha
f(t)g^\sigma(t)}\Delta t\\
+\frac{\gamma}{\Gamma(\gamma+1)}g(a)\left(\int_{a}^{b}(t+\gamma
h-a)_h^{(\gamma-1)}f(t)\Delta t -\int_{\sigma(a)}^{b}(t+\gamma
h-\sigma(a))_h^{(\gamma-1)}f(t)\Delta t\right).
\end{multline}
\end{theorem}
\begin{proof}
By \eqref{naosei1} we can write
\begin{equation}
\label{rui0}
\begin{split}
\int_{a}^{b} &f(t)_a\Delta_h^\alpha g(t)\Delta t
=\int_{a}^{b}f(t)(_a\Delta_h^{-\gamma} g(t+\gamma h))^{\Delta}\Delta t\\
&=\int_{a}^{b}f(t)\left[_a\Delta_h^{-\gamma}
g^{\Delta}(t+\gamma h)+\frac{\gamma}{\Gamma(\gamma+1)}(t+\gamma h-a)_h^{(\gamma-1)}g(a)\right]\Delta t\\
&=\int_{a}^{b}f(t)_a\Delta_h^{-\gamma}g^{\Delta}(t+\gamma h)\Delta t
+\int_{a}^{b}\frac{\gamma}{\Gamma(\gamma+1)}(t+\gamma h-a)_h^{(\gamma-1)}f(t)g(a)\Delta t.
\end{split}
\end{equation}
Using \eqref{seila1} we get
\begin{equation*}
\begin{split}
\int_{a}^{b} &f(t)_a\Delta_h^{-\gamma} g^{\Delta}(t+\gamma h) \Delta t\\
&=\int_{a}^{b}f(t)\left[h^\gamma g^{\Delta}(t)
+ \frac{\gamma}{\Gamma(\gamma+1)}\int_{a}^{t}(t+\gamma h
-\sigma(s))_h^{(\gamma-1)} g^{\Delta}(s)\Delta s\right]\Delta t\\
&=h^\gamma\int_{a}^{b}f(t)g^{\Delta}(t)\Delta
t+\frac{\gamma}{\Gamma(\gamma+1)}\int_{a}^{\rho(b)}
g^{\Delta}(t)\int_{\sigma(t)}^{b}(s+\gamma h-\sigma(t))_h^{(\gamma-1)}f(s)\Delta s \Delta t\\
&=h^\gamma f(\rho(b))[g(b)-g(\rho(b))]+\int_{a}^{\rho(b)}
g^{\Delta}(t)_h\Delta_{\rho(b)}^{-\gamma} f(t-\gamma h)\Delta t,
\end{split}
\end{equation*}
where the third equality follows by Lemma~\ref{lem1}. We proceed to develop the right hand side of the last equality as follows:
\begin{equation*}
\begin{split}
h^\gamma & f(\rho(b))[g(b)-g(\rho(b))]+\int_{a}^{\rho(b)}
g^{\Delta}(t)_h\Delta_{\rho(b)}^{-\gamma} f(t-\gamma h)\Delta t\\
&=h^\gamma f(\rho(b))[g(b)-g(\rho(b))] +\left[g(t)_h\Delta_{\rho(b)}^{-\gamma}
f(t-\gamma h)\right]_{t=a}^{t=\rho(b)}\\
&\quad -\int_{a}^{\rho(b)} g^\sigma(t)(_h\Delta_{\rho(b)}^{-\gamma} f(t-\gamma h))^{\Delta}\Delta t\\
&=h^\gamma f(\rho(b))g(b)-h^\gamma f(a)g(a)\\
&\quad -\frac{\gamma}{\Gamma(\gamma+1)}g(a)\int_{\sigma(a)}^{b}(s+\gamma h-\sigma(a))_h^{(\gamma-1)}f(s)\Delta s
+\int_{a}^{\rho(b)}{\left(_h\Delta_{\rho(b)}^\alpha f(t)\right)g^\sigma(t)}\Delta t,
\end{split}
\end{equation*}
where the first equality follows from Lemma~\ref{integracao:partes}.
Putting this into (\ref{rui0}) we get \eqref{delf:sumPart}.
\end{proof}


\subsection{Necessary optimality conditions}

We begin to fix two arbitrary real numbers $\alpha$ and $\beta$ such
that $\alpha,\beta\in(0,1]$. Further, we put $\gamma := 1-\alpha$
and $\nu :=1-\beta$.

Let a function
$L(t,u,v,w):\mathbb{T}^\kappa\times\mathbb{R}\times\mathbb{R}\times\mathbb{R}\rightarrow\mathbb{R}$
be given. We consider the problem of minimizing (or maximizing) a functional
$\mathcal{L}:\mathcal{F}_\mathbb{T}\rightarrow\mathbb{R}$ subject to
given boundary conditions:
\begin{equation}
\label{naosei7}
\mathcal{L}(y(\cdot))=\int_{a}^{b}L(t,y^{\sigma}(t),{_a}\Delta_h^\alpha
y(t),{_h}\Delta_b^\beta y(t))\Delta t \longrightarrow \min,
\ y(a)=A, \ y(b)=B \, .
\end{equation}
Our main aim is to derive necessary optimality
conditions for problem \eqref{naosei7}.
\begin{definition}
For $f\in\mathcal{F}_\mathbb{T}$ we define the norm
$$\|f\|=\max_{t\in\mathbb{T}^\kappa}|f^\sigma(t)|+\max_{t\in\mathbb{T}^\kappa}|_a\Delta_h^\alpha
f(t)|+\max_{t\in\mathbb{T}^\kappa}|_h\Delta_b^\beta f(t)|.$$
A function $\hat{y}\in\mathcal{F}_\mathbb{T}$ with $\hat{y}(a)=A$ and
$\hat{y}(b)=B$ is called a local minimum for problem
\eqref{naosei7} provided there exists $\delta>0$ such that
$\mathcal{L}(\hat{y})\leq\mathcal{L}(y)$ for all $y\in\mathcal{F}_\mathbb{T}$
with $y(a)=A$ and $y(b)=B$ and $\|y-\hat{y}\|<\delta$.
\end{definition}

\begin{definition}
A function $\eta\in\mathcal{F}_\mathbb{T}$ is called an admissible variation
provided $\eta \neq 0$ and $\eta(a)=\eta(b)=0$.
\end{definition}

From now on we assume that the second-order partial
derivatives $L_{uu}$, $L_{uv}$, $L_{uw}$, $L_{vw}$,
$L_{vv}$, and $L_{ww}$ exist and are continuous.


\subsubsection{First order optimality condition}

Next theorem gives a first order necessary condition for
problem \eqref{naosei7}, \textrm{i.e.}, an Euler-Lagrange
type equation for the fractional $h$-difference setting.
\begin{theorem}[The $h$-fractional Euler-Lagrange equation for problem \eqref{naosei7}]
\label{thm0}
If $\hat{y}\in\mathcal{F}_\mathbb{T}$ is a
local minimum for problem \eqref{naosei7}, then the
equality
\begin{equation}
\label{EL}
L_u[\hat{y}](t) +{_h}\Delta_{\rho(b)}^\alpha
L_v[\hat{y}](t)+{_a}\Delta_h^\beta L_w[\hat{y}](t)=0
\end{equation}
holds for all $t\in\mathbb{T}^{\kappa^2}$ with operator $[\cdot]$ defined by
$[y](s) =(s,y^{\sigma}(s),{_a}\Delta_s^\alpha y(s),{_s}\Delta_b^\beta y(s))$.
\end{theorem}
\begin{proof}
Suppose that $\hat{y}(\cdot)$ is a local minimum of
$\mathcal{L}[\cdot]$. Let $\eta(\cdot)$ be an arbitrarily fixed
admissible variation and define a function
$\Phi:\left(-\frac{\delta}{\|\eta(\cdot)\|},\frac{\delta}{\|\eta(\cdot)\|}\right)\rightarrow\mathbb{R}$
by
\begin{equation}
\label{fi}
\Phi(\varepsilon)=\mathcal{L}[\hat{y}(\cdot)+\varepsilon\eta(\cdot)].
\end{equation}
This function has a minimum at $\varepsilon=0$, so we must have
$\Phi'(0)=0$, i.e.,
$$\int_{a}^{b}\left[L_u[\hat{y}](t)\eta^\sigma(t)
+L_v[\hat{y}](t){_a}\Delta_h^\alpha\eta(t)
+L_w[\hat{y}](t){_h}\Delta_b^\beta\eta(t)\right]\Delta t=0,$$ which we may
write equivalently as
\begin{multline}
\label{rui3}
h L_u[\hat{y}](t)\eta^\sigma(t)|_{t=\rho(b)}+\int_{a}^{\rho(b)}L_u[\hat{y}](t)\eta^\sigma(t)\Delta
t +\int_{a}^{b}L_v[\hat{y}](t){_a}\Delta_h^\alpha\eta(t)\Delta
t\\+\int_{a}^{b}L_w[\hat{y}](t){_h}\Delta_b^\beta\eta(t)\Delta t=0.
\end{multline}
Using Theorem~\ref{teor1} and the fact that $\eta(a)=\eta(b)=0$, we get
\begin{equation}
\label{naosei5}
\int_{a}^{b}L_v[\hat{y}](t){_a}\Delta_h^\alpha\eta(t)\Delta
t=\int_{a}^{\rho(b)}\left({_h}\Delta_{\rho(b)}^\alpha
\left(L_v[\hat{y}]\right)(t)\right)\eta^\sigma(t)\Delta t
\end{equation}
for the third term in \eqref{rui3}.
Using \eqref{naosei12} it follows that
\begin{equation}
\label{naosei4}
\begin{split}
\int_{a}^{b} & L_w[\hat{y}](t){_h}\Delta_b^\beta\eta(t)\Delta t\\=&-\int_{a}^{b}L_w[\hat{y}](t)({_h}\Delta_b^{-\nu}\eta(t-\nu h))^{\Delta}\Delta t\\
=&-\int_{a}^{b}L_w[\hat{y}](t)\left[{_h}\Delta_{\rho(b)}^{-\nu}
\eta^{\Delta}(t-\nu h)-\frac{\nu}{\Gamma(\nu+1)}(b+\nu h-\sigma(t))_h^{(\nu-1)}\eta(b)\right]\Delta t\\
=&-\int_{a}^{b}L_w[\hat{y}](t){_h}\Delta_{\rho(b)}^{-\nu}
\eta^{\Delta}(t-\nu h)\Delta t
+\frac{\nu\eta(b)}{\Gamma(\nu+1)}\int_{a}^{b}(b+\nu
h-\sigma(t))_h^{(\nu-1)}L_w[\hat{y}](t)\Delta t .
\end{split}
\end{equation}
We now use Lemma~\ref{lem1} to get
\begin{equation}
\label{naosei2}
\begin{split}
\int_{a}^{b} &L_w[\hat{y}](t){_h}\Delta_{\rho(b)}^{-\nu} \eta^{\Delta}(t-\nu h)\Delta t\\
&=\int_{a}^{b}L_w[\hat{y}](t)\left[h^\nu\eta^{\Delta}(t)+\frac{\nu}{\Gamma(\nu+1)}\int_{\sigma(t)}^{b}(s+\nu
h-\sigma(t))_h^{(\nu-1)} \eta^{\Delta}(s)\Delta s\right]\Delta t\\
&=\int_{a}^{b}h^\nu L_w[\hat{y}](t)\eta^{\Delta}(t)\Delta t\\
&\qquad +\frac{\nu}{\Gamma(\nu+1)}\int_{a}^{\rho(b)}\left[L_w[\hat{y}](t)\int_{\sigma(t)}^{b}(s+\nu
h-\sigma(t))_h^{(\nu-1)} \eta^{\Delta}(s)\Delta s\right]\Delta t\\
&=\int_{a}^{b}h^\nu L_w[\hat{y}](t)\eta^{\Delta}(t)\Delta t\\
&\qquad +\frac{\nu}{\Gamma(\nu+1)}\int_{a}^{b}\left[\eta^{\Delta}(t)\int_{a}^{t}(t+\nu h
-\sigma(s))_h^{(\nu-1)}L_w[\hat{y}](s)\Delta s\right]\Delta t\\
&=\int_{a}^{b}\eta^{\Delta}(t){_a}\Delta^{-\nu}_h \left(L_w[\hat{y}]\right)(t+\nu h)\Delta t.
\end{split}
\end{equation}
We apply again the time scale integration by parts formula
(Lemma~\ref{integracao:partes}), this
time to \eqref{naosei2}, to obtain,
\begin{equation}
\label{naosei3}
\begin{split}
\int_{a}^{b} & \eta^{\Delta}(t){_a}\Delta^{-\nu}_h
\left(L_w[\hat{y}]\right)(t+\nu h)\Delta t\\
&=\int_{a}^{\rho(b)}\eta^{\Delta}(t){_a}\Delta^{-\nu}_h
\left(L_w[\hat{y}]\right)(t+\nu h)\Delta t\\
&\qquad +(\eta(b)-\eta(\rho(b))){_a}\Delta^{-\nu}_h
\left(L_w[\hat{y}]\right)(t+\nu h)|_{t=\rho(b)}\\
&=\left[\eta(t){_a}\Delta^{-\nu}_h
\left(L_w[\hat{y}]\right)(t+\nu h)\right]_{t=a}^{t=\rho(b)}
-\int_{a}^{\rho(b)}\eta^\sigma(t)({_a}\Delta^{-\nu}_h
\left(L_w[\hat{y}]\right)(t+\nu h))^\Delta \Delta t\\
&\qquad +\eta(b){_a}\Delta^{-\nu}_h
\left(L_w[\hat{y}]\right)(t+\nu h)|_{t=\rho(b)}-\eta(\rho(b)){_a}\Delta^{-\nu}_h
\left(L_w[\hat{y}]\right)(t+\nu h)|_{t=\rho(b)}\\
&=\eta(b){_a}\Delta^{-\nu}_h
\left(L_w[\hat{y}]\right)(t+\nu h)|_{t=\rho(b)}-\eta(a){_a}\Delta^{-\nu}_h
\left(L_w[\hat{y}]\right)(t+\nu h)|_{t=a}\\
&\qquad -\int_{a}^{\rho(b)}\eta^\sigma(t){_a}\Delta^{\beta}_h
\left(L_w[\hat{y}]\right)(t)\Delta t.
\end{split}
\end{equation}
Since $\eta(a)=\eta(b)=0$ we obtain, from \eqref{naosei2} and
\eqref{naosei3}, that
$$\int_{a}^{b}L_w[\hat{y}](t){_h}\Delta_{\rho(b)}^{-\nu}
\eta^\Delta(t)\Delta t
=-\int_{a}^{\rho(b)}\eta^\sigma(t){_a}\Delta^{\beta}_h
\left(L_w[\hat{y}]\right)(t)\Delta t\, ,$$
and after inserting in \eqref{naosei4}, that
\begin{equation}
\label{naosei6}
\int_{a}^{b}L_w[\hat{y}](t){_h}\Delta_b^\beta\eta(t)\Delta t
=\int_{a}^{\rho(b)}\eta^\sigma(t){_a}\Delta^{\beta}_h
\left(L_w[\hat{y}]\right)(t) \Delta t.
\end{equation}
By \eqref{naosei5} and \eqref{naosei6} we may write \eqref{rui3} as
$$\int_{a}^{\rho(b)}\left[L_u[\hat{y}](t)
+{_h}\Delta_{\rho(b)}^\alpha \left(L_v[\hat{y}]\right)(t)+{_a}\Delta_h^\beta
\left(L_w[\hat{y}]\right)(t)\right]\eta^\sigma(t) \Delta t =0\, .$$
Since the values of $\eta^\sigma(t)$ are arbitrary for
$t\in\mathbb{T}^{\kappa^2}$, the Euler-Lagrange equation \eqref{EL}
holds along $\hat{y}$.
\end{proof}

The next result is a direct corollary of Theorem~\ref{thm0}.

\begin{corollary}[The $h$-Euler-Lagrange equation
-- \textrm{cf.}, \textrm{e.g.}, \cite{CD:Bohner:2004,RD}]
\label{ELCor}
Let $\mathbb{T}$ be the time scale $h \mathbb{Z}$, $h > 0$, with
the forward jump operator $\sigma$ and the delta derivative $\Delta$.
Assume $a, b \in \mathbb{T}$, $a < b$. If $\hat{y}$ is a solution
to the problem
\begin{equation*}
\mathcal{L}(y(\cdot))=\int_{a}^{b}L(t,y^{\sigma}(t),y^\Delta(t))\Delta t \longrightarrow \min,
\  y(a)=A, \  y(b)=B\, ,
\end{equation*}
then the equality
$L_u(t,\hat{y}^{\sigma}(t),\hat{y}^\Delta(t))-\left(L_v(t,\hat{y}^{\sigma}(t),\hat{y}^\Delta(t))\right)^\Delta =0$
holds for all $t\in\mathbb{T}^{\kappa^2}$.
\end{corollary}
\begin{proof}
Choose $\alpha=1$ and a $L$ that does not depend on $w$
in Theorem~\ref{thm0}.
\end{proof}

\begin{remark}
If we take $h=1$ in Corollary~\ref{ELCor} we have that
$$L_u(t,\hat{y}^{\sigma}(t),\Delta\hat{y}(t))-\Delta L_v(t,\hat{y}^{\sigma}(t),\Delta\hat{y}(t)) =0$$
holds for all $t\in\mathbb{T}^{\kappa^2}$.
This equation is usually called \emph{the discrete Euler-Lagrange equation},
and can be found, \textrm{e.g.}, in \cite[Chap.~8]{book:DCV}.
\end{remark}


\subsubsection{Natural boundary conditions}

If the initial condition $y(a)=A$ is not present
in problem \eqref{naosei7} (\textrm{i.e.}, $y(a)$ is free),
besides the $h$-fractional Euler-Lagrange equation \eqref{EL}
the following supplementary condition must be fulfilled:
\begin{multline}\label{rui1}
-h^\gamma L_v[\hat{y}](a)+\frac{\gamma}{\Gamma(\gamma+1)}\left(
\int_{a}^{b}(t+\gamma h-a)_h^{(\gamma-1)}L_v[\hat{y}](t)\Delta t\right.\\
\left.-\int_{\sigma(a)}^{b}(t+\gamma
h-\sigma(a))_h^{(\gamma-1)}L_v[\hat{y}](t)\Delta t\right)+ L_w[\hat{y}](a)=0.
\end{multline}
Similarly, if $y(b)=B$ is not present in \eqref{naosei7} ($y(b)$ is free), the
extra condition
\begin{multline}\label{rui2}
h L_u[\hat{y}](\rho(b))+h^\gamma L_v[\hat{y}](\rho(b))-h^\nu L_w[\hat{y}](\rho(b))\\
+\frac{\nu}{\Gamma(\nu+1)}\left(\int_{a}^{b}(b+\nu
h-\sigma(t))_h^{(\nu-1)}L_w[\hat{y}](t)\Delta t \right.\\ \left.
-\int_{a}^{\rho(b)}(\rho(b)+\nu
h-\sigma(t))_h^{(\nu-1)}L_w[\hat{y}](t)\Delta t\right)=0
\end{multline}
is added to Theorem~\ref{thm0}. We leave the proof
of the \emph{natural boundary conditions}
\eqref{rui1} and \eqref{rui2} to the reader.
We just note here that the first term in \eqref{rui2}
arises from the first term of the left hand side of \eqref{rui3}.


\subsubsection{Second order optimality condition}

We now obtain a second order necessary condition for problem
\eqref{naosei7}, \textrm{i.e.}, we prove a Legendre optimality
type condition for the fractional $h$-difference setting.
\begin{theorem}[The $h$-fractional Legendre necessary condition]
\label{thm1}
If $\hat{y}\in\mathcal{F}_\mathbb{T}$ is a local minimum for problem
\eqref{naosei7}, then the inequality
\begin{equation}
\label{eq:LC}
\begin{split}
h^2 &L_{uu}[\hat{y}](t)+2h^{\gamma+1}L_{uv}[\hat{y}](t)+2h^{\nu+1}(\nu-1)L_{uw}[\hat{y}](t)
+h^{2\gamma}(\gamma -1)^2 L_{vv}[\hat{y}](\sigma(t))\\
&+2h^{\nu+\gamma}(\gamma-1)L_{vw}[\hat{y}](\sigma(t))+2h^{\nu+\gamma}(\nu-1)L_{vw}[\hat{y}](t)+h^{2\nu}(\nu-1)^2 L_{ww}[\hat{y}](t)\\
&+h^{2\nu}L_{ww}[\hat{y}](\sigma(t))
+\int_{a}^{t}h^3L_{ww}[\hat{y}](s)\left(\frac{\nu(1-\nu)}{\Gamma(\nu+1)}(t+\nu
h - \sigma(s))_h^{(\nu-2)}\right)^2\Delta s\\
&+h^{\gamma}L_{vv}[\hat{y}](t)
+\int_{\sigma(\sigma(t))}^{b}h^3L_{vv}[\hat{y}](s)\left(\frac{\gamma(\gamma-1)}{\Gamma(\gamma+1)}(s+\gamma h
-\sigma(\sigma(t)))_h^{(\gamma-2)}\right)^2\Delta s \geq 0
\end{split}
\end{equation}
holds for all $t\in\mathbb{T}^{\kappa^2}$, where
$[\hat{y}](t)=(t,\hat{y}^{\sigma}(t),{_a}\Delta_t^\alpha
\hat{y}(t),{_t}\Delta_b^\beta\hat{y}(t))$.
\end{theorem}
\begin{proof}
By the hypothesis of the theorem, and letting $\Phi$ be as in
\eqref{fi}, we have as necessary optimality condition that $\Phi''(0)\geq 0$
for an arbitrary admissible variation $\eta(\cdot)$.
Inequality $\Phi''(0)\geq 0$ is equivalent to
\begin{multline}
\label{des1}
\int_{a}^{b}\left[L_{uu}[\hat{y}](t)(\eta^\sigma(t))^2
+2L_{uv}[\hat{y}](t)\eta^\sigma(t){_a}\Delta_h^\alpha\eta(t)
+2L_{uw}[\hat{y}](t)\eta^\sigma(t){_h}\Delta_b^\beta\eta(t)\right.\\
\left. +L_{vv}[\hat{y}](t)({_a}\Delta_h^\alpha\eta(t))^2
+2L_{vw}[\hat{y}](t){_a}\Delta_h^\alpha\eta(t){_h}\Delta_b^\beta\eta(t)
+L_{ww}(t)({_h}\Delta_b^\beta\eta(t))^2\right]\Delta t\geq 0.
\end{multline}
Let $\tau\in\mathbb{T}^{\kappa^2}$ be arbitrary, and choose
$\eta:\mathbb{T}\rightarrow\mathbb{R}$ given by
$\eta(t) = \left\{ \begin{array}{ll}
h & \mbox{if $t=\sigma(\tau)$};\\
0 & \mbox{otherwise}.\end{array} \right.$ It follows that
$\eta(a)=\eta(b)=0$, \textrm{i.e.}, $\eta$ is an admissible variation.
Using \eqref{naosei1} we get
\begin{equation*}
\begin{split}
\int_{a}^{b}&\left[L_{uu}[\hat{y}](t)(\eta^\sigma(t))^2
+2L_{uv}[\hat{y}](t)\eta^\sigma(t){_a}\Delta_h^\alpha\eta(t)
+L_{vv}[\hat{y}](t)({_a}\Delta_h^\alpha\eta(t))^2\right]\Delta t\\
&=\int_{a}^{b}\Biggl[L_{uu}[\hat{y}](t)(\eta^\sigma(t))^2\\
&\qquad\quad +2L_{uv}[\hat{y}](t)\eta^\sigma(t)\left(h^\gamma \eta^{\Delta}(t)+
\frac{\gamma}{\Gamma(\gamma+1)}\int_{a}^{t}(t+\gamma h
-\sigma(s))_h^{(\gamma-1)}\eta^{\Delta}(s)\Delta s\right)\\
&\qquad\quad +L_{vv}[\hat{y}](t)\left(h^\gamma \eta^{\Delta}(t)
+\frac{\gamma}{\Gamma(\gamma+1)}\int_{a}^{t}(t+\gamma h
-\sigma(s))_h^{(\gamma-1)}\eta^{\Delta}(s)\Delta s\right)^2\Biggr]\Delta t\\
&=h^3L_{uu}[\hat{y}](\tau)+2h^{\gamma+2}L_{uv}[\hat{y}](\tau)+h^{\gamma+1}L_{vv}[\hat{y}](\tau)\\
&\quad +\int_{\sigma(\tau)}^{b}L_{vv}[\hat{y}](t)\left(h^\gamma\eta^{\Delta}(t)
+\frac{\gamma}{\Gamma(\gamma+1)}\int_{a}^{t}(t+\gamma
h-\sigma(s))_h^{(\gamma-1)}\eta^{\Delta}(s)\Delta s\right)^2\Delta t.
\end{split}
\end{equation*}
Observe that
\begin{multline*}
h^{2\gamma+1}(\gamma -1)^2 L_{vv}[\hat{y}](\sigma(\tau))\\
+\int_{\sigma^2(\tau)}^{b}L_{vv}[\hat{y}](t)\left(\frac{\gamma}{\Gamma(\gamma+1)}\int_{a}^{t}(t+\gamma
h-\sigma(s))_h^{(\gamma-1)}\eta^{\Delta}(s)\Delta s\right)^2\Delta t\\
=\int_{\sigma(\tau)}^{b}L_{vv}[\hat{y}](t)\left(h^\gamma
\eta^\Delta(t)+\frac{\gamma}{\Gamma(\gamma+1)}\int_{a}^{t}(t+\gamma
h-\sigma(s))_h^{(\gamma-1)}\eta^{\Delta}(s)\Delta s\right)^2\Delta t.
\end{multline*}
Let $t\in[\sigma^2(\tau),\rho(b)]\cap h\mathbb{Z}$. Since
\begin{equation}
\label{rui10}
\begin{split}
\frac{\gamma}{\Gamma(\gamma+1)}&\int_{a}^{t}(t+\gamma h -\sigma(s))_h^{(\gamma-1)}\eta^{\Delta}(s)\Delta s\\
&= \frac{\gamma}{\Gamma(\gamma+1)}\left[\int_{a}^{\sigma(\tau)}(t+\gamma h-\sigma(s))_h^{(\gamma-1)}\eta^{\Delta}(s)\Delta s\right.\\
&\qquad\qquad\qquad\qquad \left.+\int_{\sigma(\tau)}^{t}(t+\gamma h-\sigma(s))_h^{(\gamma-1)}\eta^{\Delta}(s)\Delta s\right]\\
&=h\frac{\gamma}{\Gamma(\gamma+1)}\left[(t+\gamma h-\sigma(\tau))_h^{(\gamma-1)}-(t+\gamma h-\sigma(\sigma(\tau)))_h^{(\gamma-1)}\right]\\
&=\frac{\gamma h^\gamma}{\Gamma(\gamma+1)}\left[
\frac{\left(\frac{t-\tau}{h}+\gamma-1\right)\Gamma\left(\frac{t-\tau}{h}+\gamma-1\right)
-\left(\frac{t-\tau}{h}\right)\Gamma\left(\frac{t-\tau}{h}+\gamma-1\right)}
{\left(\frac{t-\tau}{h}\right)\Gamma\left(\frac{t-\tau}{h}\right)}\right]\\
&=h^{2}\frac{\gamma(\gamma-1)}{\Gamma(\gamma+1)}(t+\gamma h
-\sigma(\sigma(\tau)))_h^{(\gamma-2)},
\end{split}
\end{equation}
we conclude that
\begin{multline*}
\int_{\sigma^2(\tau)}^{b}L_{vv}[\hat{y}](t)\left(\frac{\gamma}{\Gamma(\gamma+1)}\int_{a}^{t}(t
+\gamma h-\sigma(s))_h^{(\gamma-1)}\eta^{\Delta}(s)\Delta s\right)^2\Delta t\\
=\int_{\sigma^2(\tau)}^{b}L_{vv}[\hat{y}](t)\left(h^2\frac{\gamma(\gamma-1)}{\Gamma(\gamma+1)}(t
+\gamma h-\sigma^2(\tau))_h^{(\gamma-2)}\right)^2\Delta t.
\end{multline*}
Note that we can write
${_t}\Delta_b^\beta\eta(t)=-{_h}\Delta_{\rho(b)}^{-\nu}
\eta^\Delta(t-\nu h)$ because $\eta(b)=0$.
It is not difficult to see that the following equality holds:
\begin{equation*}
\begin{split}
\int_{a}^{b}2L_{uw}[\hat{y}](t)\eta^\sigma(t){_h}\Delta_b^\beta\eta(t)\Delta t
&=-\int_{a}^{b}2L_{uw}[\hat{y}](t)\eta^\sigma(t){_h}\Delta_{\rho(b)}^{-\nu}
\eta^\Delta(t-\nu h)\Delta t\\
&=2h^{2+\nu}L_{uw}[\hat{y}](\tau)(\nu-1) \, .
\end{split}
\end{equation*}
Moreover,
\begin{equation*}
\begin{split}
\int_{a}^{b} &2L_{vw}[\hat{y}](t){_a}\Delta_h^\alpha\eta(t){_h}\Delta_b^\beta\eta(t)\Delta t\\
&=-2\int_{a}^{b}L_{vw}[\hat{y}](t)\left\{\left(h^\gamma\eta^{\Delta}(t)+\frac{\gamma}{\Gamma(\gamma+1)}
\cdot\int_{a}^{t}(t+\gamma h-\sigma(s))_h^{(\gamma-1)}\eta^{\Delta}(s)\Delta s\right)\right.\\
&\qquad\qquad \left.\cdot\left[h^\nu\eta^{\Delta}(t)+\frac{\nu}{\Gamma(\nu+1)}\int_{\sigma(t)}^{b}(s
+\nu h-\sigma(t))_h^{(\nu-1)}\eta^{\Delta}(s)\Delta s\right]\right\}\Delta t\\
&=2h^{\gamma+\nu+1}(\nu-1)L_{vw}[\hat{y}](\tau)+2h^{\gamma+\nu+1}(\gamma-1)L_{vw}[\hat{y}](\sigma(\tau)).
\end{split}
\end{equation*}
Finally, we have that
\begin{equation*}
\begin{split}
&\int_{a}^{b} L_{ww}[\hat{y}](t)({_h}\Delta_b^\beta\eta(t))^2\Delta t\\
&=\int_{a}^{\sigma(\sigma(\tau))}L_{ww}[\hat{y}](t)\left[h^\nu\eta^{\Delta}(t)+\frac{\nu}{\Gamma(\nu+1)}
\int_{\sigma(t)}^{b}(s+\nu h-\sigma(t))_h^{(\nu-1)}\eta^{\Delta}(s)\Delta s\right]^2\Delta t\\
&=\int_{a}^{\tau}L_{ww}[\hat{y}](t)\left[\frac{\nu}{\Gamma(\nu+1)}\int_{\sigma(t)}^{b}(s
+\nu h-\sigma(t))_h^{(\nu-1)}\eta^{\Delta}(s)\Delta s\right]^2\Delta t\\
&\qquad +hL_{ww}[\hat{y}](\tau)(h^\nu-\nu h^\nu)^2+h^{2\nu+1}L_{ww}[\hat{y}](\sigma(\tau))\\
&=\int_{a}^{\tau}L_{ww}[\hat{y}](t)\left[h\frac{\nu}{\Gamma(\nu+1)}\left\{(\tau+\nu h
-\sigma(t))_h^{(\nu-1)}-(\sigma(\tau)+\nu h-\sigma(t))_h^{(\nu-1)}\right\}\right]^2\\
&\qquad + hL_{ww}[\hat{y}](\tau)(h^\nu-\nu h^\nu)^2+h^{2\nu+1}L_{ww}[\hat{y}](\sigma(\tau)).
\end{split}
\end{equation*}
Similarly as we did in \eqref{rui10}, we can prove that
\begin{multline*}
h\frac{\nu}{\Gamma(\nu+1)}\left\{(\tau+\nu
h-\sigma(t))_h^{(\nu-1)}-(\sigma(\tau)+\nu
h-\sigma(t))_h^{(\nu-1)}\right\}\\
=h^{2}\frac{\nu(1-\nu)}{\Gamma(\nu+1)}(\tau+\nu h-\sigma(t))_h^{(\nu-2)}.
\end{multline*}
Thus, we have that inequality \eqref{des1} is equivalent to
\begin{multline}
\label{des2}
h\Biggl\{h^2L_{uu}[\hat{y}](t)+2h^{\gamma+1}L_{uv}[\hat{y}](t)
+h^{\gamma}L_{vv}[\hat{y}](t)+L_{vv}(\sigma(t))(\gamma h^\gamma-h^\gamma)^2\\
+\int_{\sigma(\sigma(t))}^{b}h^3L_{vv}(s)\left(\frac{\gamma(\gamma-1)}{\Gamma(\gamma+1)}(s
+\gamma h -\sigma(\sigma(t)))_h^{(\gamma-2)}\right)^2\Delta s\\
+2h^{\nu+1}L_{uw}[\hat{y}](t)(\nu-1)+2h^{\gamma+\nu}(\nu-1)L_{vw}[\hat{y}](t)\\
+2h^{\gamma+\nu}(\gamma-1)L_{vw}(\sigma(t))+h^{2\nu}L_{ww}[\hat{y}](t)(1-\nu)^2+h^{2\nu}L_{ww}[\hat{y}](\sigma(t))\\
+\int_{a}^{t}h^3L_{ww}[\hat{y}](s)\left(\frac{\nu(1-\nu)}{\Gamma(\nu+1)}(t+\nu h
- \sigma(s))^{\nu-2}\right)^2\Delta s\Biggr\}\geq 0.
\end{multline}
Because $h>0$, \eqref{des2} is equivalent to \eqref{eq:LC}.
The theorem is proved.
\end{proof}

The next result is a simple corollary of Theorem~\ref{thm1}.
\begin{corollary}[The $h$-Legendre necessary condition -- \textrm{cf.} Result~1.3 of \cite{CD:Bohner:2004}]
\label{CorDis:Bohner}
Let $\mathbb{T}$ be the time scale $h \mathbb{Z}$, $h > 0$, with
the forward jump operator $\sigma$ and the delta derivative $\Delta$.
Assume $a, b \in \mathbb{T}$, $a < b$. If $\hat{y}$ is a solution
to the problem
\begin{equation*}
\mathcal{L}(y(\cdot))=\int_{a}^{b}L(t,y^{\sigma}(t),y^\Delta(t))\Delta t \longrightarrow \min,
\  y(a)=A, \ y(b)=B \, ,
\end{equation*}
then the inequality
\begin{equation}
\label{LNCBohner}
h^2L_{uu}[\hat{y}](t)+2hL_{uv}[\hat{y}](t)+L_{vv}[\hat{y}](t)+L_{vv}[\hat{y}](\sigma(t)) \geq 0
\end{equation}
holds for all $t\in\mathbb{T}^{\kappa^2}$, where
$[\hat{y}](t)=(t,\hat{y}^{\sigma}(t),\hat{y}^\Delta(t))$.
\end{corollary}
\begin{proof}
Choose $\alpha=1$ and a Lagrangian $L$ that does not depend on $w$.
Then, $\gamma=0$ and the result follows immediately from Theorem~\ref{thm1}.
\end{proof}

\begin{remark}
When $h$ goes to zero we have $\sigma(t) = t$
and inequality \eqref{LNCBohner} coincides with
Legendre's classical necessary optimality condition $L_{vv}[\hat{y}](t) \ge 0$
(\textrm{cf.}, \textrm{e.g.}, \cite{vanBrunt}).
\end{remark}


\section{Examples}
\label{sec2}

In this section we present some illustrative examples.

\begin{example}
\label{ex:2} Let us consider the following problem:
\begin{equation}
\label{eq:ex2} \mathcal{L}(y)=\frac{1}{2}
\int_{0}^{1} \left({_0}\Delta_h^{\frac{3}{4}} y(t)\right)^2\Delta t
\longrightarrow \min \, , \quad y(0)=0 \, , \quad y(1)=1 \, .
\end{equation}
We consider (\ref{eq:ex2}) with different
values of $h$. Numerical results show that when $h$ tends to
zero the $h$-fractional Euler-Lagrange extremal tends to the fractional continuous extremal:
when $h \rightarrow 0$ (\ref{eq:ex2}) tends to the
fractional continuous variational problem
in the Riemann-Liouville sense studied in \cite[Example~1]{agr0},
with solution given by
\begin{equation}
\label{solEx2}
y(t)=\frac{1}{2}\int_0^t\frac{dx}{\left[(1-x)(t-x)\right]^{\frac{1}{4}}} \, .
\end{equation}
This is illustrated in Figure~\ref{Fig:2}.
\begin{figure}[ht]
\begin{center}
\includegraphics[scale=0.45]{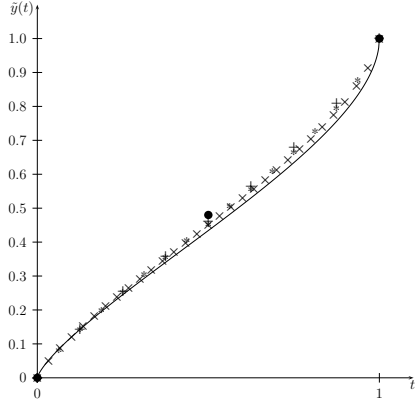}
  \caption{Extremal $\tilde{y}(t)$ for problem of Example~\ref{ex:2}
  with different values of $h$:
  $h=0.50$ ($\bullet$); $h=0.125$ ($+$);
  $h=0.0625$ ($\ast$); $h=1/30$ ($\times$).
  The continuous line represent function
  (\ref{solEx2}).}\label{Fig:2}
\end{center}
\end{figure}
In this example for each value of $h$ there is a unique
$h$-fractional Euler-Lagrange extremal,
solution of \eqref{EL}, which always verifies the
$h$-fractional Legendre necessary condition \eqref{eq:LC}.
\end{example}

\begin{example}
\label{ex:1} Let us consider the following problem:
\begin{equation}
\label{eq:ex1}
\mathcal{L}(y)=\int_{0}^{1}
\left[\frac{1}{2}\left({_0}\Delta_h^\alpha
y(t)\right)^2-y^{\sigma}(t)\right]\Delta t \longrightarrow \min \, ,
\quad y(0) = 0 \, , \quad y(1) = 0 \, .
\end{equation}
We begin by considering problem (\ref{eq:ex1}) with a fixed value for $\alpha$
and different values of $h$. The extremals $\tilde{y}$ are obtained
using our Euler-Lagrange equation (\ref{EL}).
As in Example~\ref{ex:2} the numerical results show that when $h$ tends to zero the extremal
of the problem tends to the extremal of the corresponding continuous
fractional problem of the calculus of variations in the Riemann-Liouville sense.
More precisely, when $h$ approximates zero problem (\ref{eq:ex1}) tends to the fractional
continuous problem studied in \cite[Example~2]{agr2}. For $\alpha=1$
and $h \rightarrow 0$ the extremal of (\ref{eq:ex1}) is given by
$y(t)=\frac{1}{2} t (1-t)$,
which coincides with the extremal of the classical
problem of the calculus of variations
\begin{equation*}
\mathcal{L}(y)=\int_{0}^{1} \left(\frac{1}{2} y'(t)^2-y(t)\right) dt \longrightarrow \min \, ,
\quad y(0) = 0 \, , \quad y(1) = 0 \, .
\end{equation*}
This is illustrated in Figure~\ref{Fig:0}
for $h = \frac{1}{2^i}$, $i = 1, 2, 3, 4$.
\begin{figure}[ht]
\begin{minipage}[b]{0.45\linewidth}
\begin{center}
\includegraphics[scale=0.45]{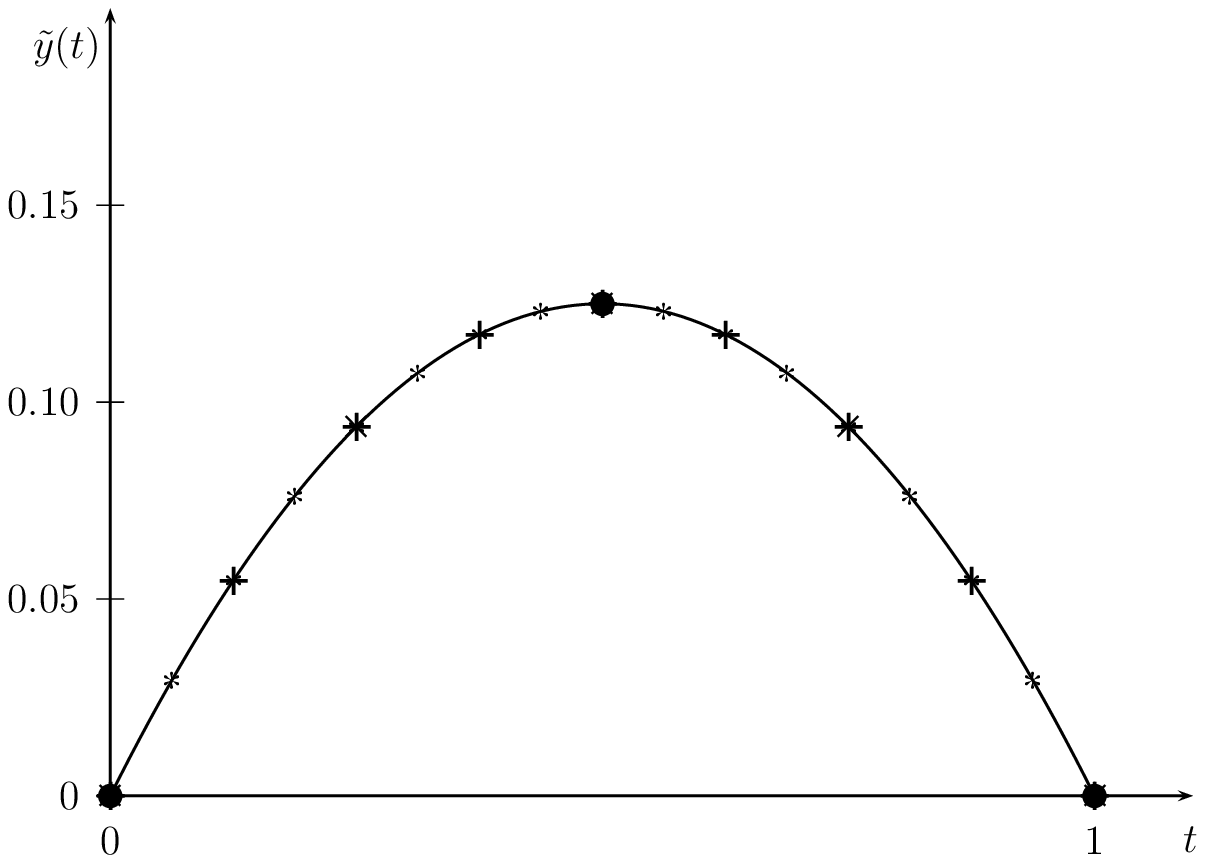}
  \caption{Extremal $\tilde{y}(t)$ for problem \eqref{eq:ex1}
  with $\alpha=1$ and different values of $h$:
  $h=0.5$ ($\bullet$); $h=0.25$ ($\times$);
  $h=0.125$ ($+$); $h=0.0625$ ($\ast$).}\label{Fig:0}
\end{center}
\end{minipage}
\hspace{0.05cm}
\begin{minipage}[b]{0.45\linewidth}
\begin{center}
\includegraphics[scale=0.45]{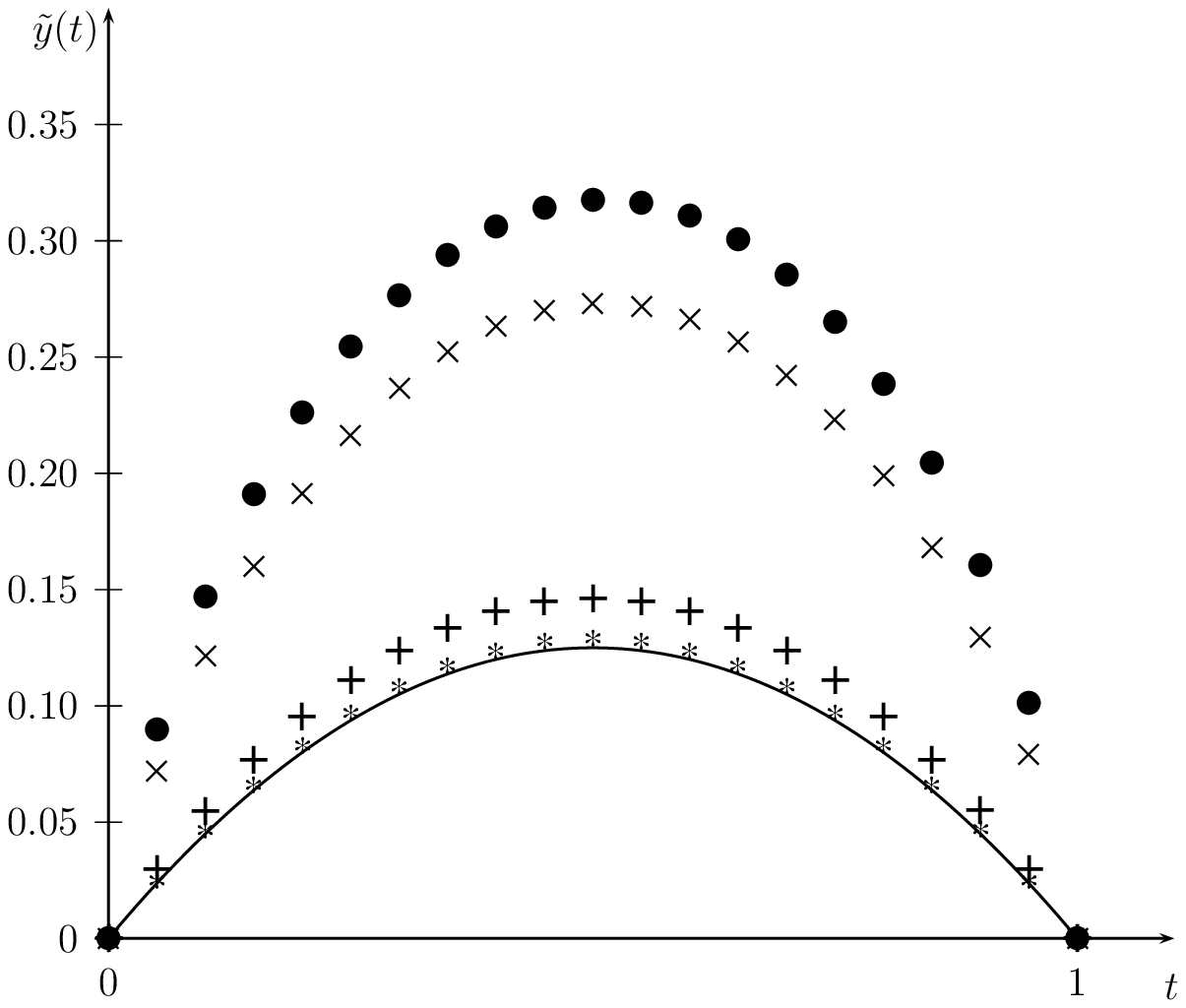}
  \caption{Extremal $\tilde{y}(t)$ for \eqref{eq:ex1}
with $h=0.05$ and different values of $\alpha$:
$\alpha=0.70$ ($\bullet$); $\alpha=0.75$ ($\times$);
$\alpha=0.95$ ($+$); $\alpha=0.99$ ($\ast$).
The continuous line is $y(t)=\frac{1}{2} t (1-t)$.}\label{Fig:1}
\end{center}
\end{minipage}
\end{figure}
In this example, for each value of $\alpha$ and $h$,
we only have one extremal (we only have one solution
to (\ref{EL}) for each $\alpha$ and $h$).
Our Legendre condition \eqref{eq:LC}
is always verified along the extremals.
Figure~\ref{Fig:1} shows the extremals of problem \eqref{eq:ex1}
for a fixed value of $h$ ($h=1/20$) and
different values of $\alpha$. The numerical results show
that when $\alpha$ tends to one the extremal tends to the solution
of the classical (integer order) discrete-time problem.
\end{example}

Our last example shows that the $h$-fractional Legendre necessary optimality condition
can be a very useful tool. In Example~\ref{ex:3} we consider a problem
for which the $h$-fractional Euler-Lagrange equation
gives several candidates but just a few of them verify
the Legendre condition \eqref{eq:LC}.

\begin{example}
\label{ex:3} Let us consider the following problem:
\begin{equation}
\label{eq:ex3} \mathcal{L}(y)=\int_{a}^{b} \left({_a}\Delta_h^\alpha
y(t)\right)^3+\theta\left({_h}\Delta_b^\alpha y(t)\right)^2\Delta t
\longrightarrow \min \, , \quad y(a)=0 \, , \quad y(b)=1 \, .
\end{equation}
For $\alpha=0.8$, $\beta=0.5$, $h=0.25$, $a=0$, $b=1$, and $\theta=1$,
problem (\ref{eq:ex3}) has eight different Euler-Lagrange extremals.
As we can see on Table~\ref{candidates:ex3}
only two of the candidates verify the Legendre condition.
To determine the best candidate we compare the values
of the functional $\mathcal{L}$ along the two good candidates.
The extremal we are looking for is given by the candidate number five on Table~\ref{candidates:ex3}.
\begin{table}
\footnotesize
\centering
\begin{tabular}{|c|c|c|c|c|c|}\hline
\# & $\tilde{y}\left(\frac{1}{4}\right)$ & $\tilde{y}\left(\frac{1}{2}\right)$
& $\tilde{y}\left(\frac{3}{4}\right)$ & $\mathcal{L}(\tilde{y})$ & Legendre condition \eqref{eq:LC}\\
\hline
         1 & -0.5511786 &  0.0515282 &  0.5133134 &  9.3035911 &         Not verified \\
\hline
         2 &  0.2669091 &  0.4878808 &  0.7151924 &  2.0084203 &        Verified \\
\hline
         3 & -2.6745703 &  0.5599360 & -2.6730125 & 698.4443232 &         Not verified \\
\hline
         4 &  0.5789976 &  1.0701515 &  0.1840377 & 12.5174960 &         Not verified \\
\hline
         5 &  1.0306820 &  1.8920322 &  2.7429222 & -32.7189756 &        Verified \\
\hline
         6 &  0.5087946 & -0.1861431 &  0.4489196 & 10.6730959 &         Not verified \\
\hline
         7 &  4.0583690 & -1.0299054 & -5.0030989 & 2451.7637948 &         Not verified \\
\hline
         8 & -1.7436106 & -3.1898449 & -0.8850511 & 238.6120299 &         Not verified \\
\hline
  \end{tabular}
\smallskip
  \caption{There exist 8 Euler-Lagrange extremals for problem \eqref{eq:ex3}
  with $\alpha=0.8$, $\beta=0.5$, $h=0.25$, $a=0$, $b=1$, and $\theta=1$,
  but only 2 of them satisfy the fractional Legendre condition \eqref{eq:LC}.}
  \label{candidates:ex3}
\end{table}
\begin{table}
\footnotesize
\centering
\begin{tabular}{|c|c|c|c|c|c|c|} \hline
\# & $\tilde{y}(0.1)$ & $\tilde{y}(0.2)$ & $\tilde{y}(0.3)$ & $\tilde{y}(0.4)$ &  $\mathcal{L}(\tilde{y})$ & \eqref{eq:LC}\\ \hline

         1 & -0.305570704 & -0.428093486 & 0.223708338 & 0.480549114 & 12.25396166 &         No \\\hline

         2 & -0.427934654 & -0.599520948 & 0.313290997 & -0.661831134 & 156.2317667 &         No \\\hline

         3 & 0.284152257 & -0.227595659 & 0.318847274 & 0.531827387 & 8.669645848 &         No \\\hline

         4 & -0.277642565 & 0.222381632 & 0.386666793 & 0.555841555 & 6.993518478 &         No \\\hline

         5 & 0.387074742 & -0.310032839 & 0.434336603 & -0.482903047 & 110.7912605 &         No \\\hline

         6 & 0.259846344 & 0.364035314 & 0.463222456 & 0.597907505 & 5.104389191 &        Yes \\\hline

         7 & -0.375094681 & 0.300437245 & 0.522386246 & -0.419053781 & 93.95316858 &         No \\\hline

         8 & 0.343327771 & 0.480989769 & 0.61204299 & -0.280908953 & 69.23497954 &         No \\\hline

         9 & 0.297792192 & 0.417196073 & -0.218013689 & 0.460556635 & 14.12227593 &         No \\\hline

        10 & 0.41283304 & 0.578364133 & -0.302235104 & -0.649232892 & 157.8272685 &         No \\\hline

        11 & -0.321401682 & 0.257431098 & -0.360644857 & 0.400971272 & 19.87468886 &         No \\\hline

        12 & 0.330157414 & -0.264444122 & -0.459803086 & 0.368850105 & 24.84475504 &         No \\\hline

        13 & -0.459640837 & 0.368155651 & -0.515763025 & -0.860276767 & 224.9964788 &         No \\\hline

        14 & -0.359429958 & -0.50354835 & -0.640748011 & 0.294083676 & 34.43515839 &         No \\\hline

        15 & 0.477760586 & -0.382668914 & -0.66536683 & -0.956478654 & 263.3075289 &         No \\\hline

        16 & -0.541587541 & -0.758744525 & -0.965476394 & -1.246195157 & 392.9592508 &         No \\\hline
\end{tabular}
\smallskip
\caption{There exist 16 Euler-Lagrange extremals for problem \eqref{eq:ex3}
  with $\alpha=0.3$, $h=0.1$, $a=0$, $b=0.5$, and $\theta=0$,
  but only 1 (candidate \#6) satisfy the fractional Legendre condition \eqref{eq:LC}.}\label{16dados}
\end{table}

For problem (\ref{eq:ex3}) with $\alpha=0.3$, $h=0.1$, $a=0$, $b=0.5$, and $\theta=0$,
we obtain the results of Table~\ref{16dados}:
there exist sixteen Euler-Lagrange extremals but only one
satisfy the fractional Legendre condition.
The extremal we are looking for is given by the candidate number six on Table~\ref{16dados}.
\end{example}

The numerical results show that the solutions to our
discrete-time fractional variational problems
converge to the classical discrete-time solutions
when the fractional order of the discrete-derivatives tend to integer values,
and to the fractional Riemann-Liouville continuous-time solutions
when $h$ tends to zero.


\section{Conclusion}
\label{sec:conc}

The discrete fractional calculus is a recent subject
under strong current development due to its importance
as a modeling tool of real phenomena.
In this work we introduce a new fractional difference variational calculus
in the time-scale $(h\mathbb{Z})_a$, $h > 0$ and $a$ a real number,
for Lagrangians depending on left and right discrete-time fractional derivatives.
Our objective was to introduce the concept of left and right fractional sum/difference
(\textrm{cf.} Definition~\ref{def0}) and to develop the theory of fractional difference calculus.
An Euler--Lagrange type equation \eqref{EL},
fractional natural boundary conditions \eqref{rui1} and \eqref{rui2},
and a second order Legendre type necessary optimality condition \eqref{eq:LC},
were obtained. The results are based on a new discrete fractional summation by parts formula
\eqref{delf:sumPart} for $(h\mathbb{Z})_a$. Obtained first and second order necessary optimality conditions
were implemented computationally in the computer algebra systems
\textsf{Maple} and \textsf{Maxima}. Our numerical results show that:
\begin{enumerate}

\item the solutions of our fractional problems converge to
the classical discrete-time solutions in $(h\mathbb{Z})_a$
when the fractional order of the discrete-derivatives tend to integer values;

\item the solutions of the considered fractional problems converge
to the fractional Riemann--Liouville continuous solutions
when $h \rightarrow 0$;

\item there are cases for which the fractional Euler--Lagrange equation
give only one candidate that does not verify the obtained
Legendre condition (so the problem at hands does not have a minimum);

\item there are cases for which the Euler--Lagrange equation give only one
candidate that verify the Legendre condition (so the extremal is
a candidate for minimizer, not for maximizer);

\item there are cases for which the Euler--Lagrange equation give us several
candidates and just a few of them verify the Legendre condition.

\end{enumerate}
We can say that the obtained Legendre condition
can be a very practical tool to conclude when a candidate
identified via the Euler--Lagrange equation
is really a solution of the fractional variational problem.
It is worth to mention that a fractional Legendre condition
for the continuous fractional variational calculus
is still an open question.

Undoubtedly, much remains to be done in the development of the theory
of discrete fractional calculus of variations in $(h\mathbb{Z})_a$
here initiated. Moreover, we trust that the present work
will initiate research not only in the area
of the discrete-time fractional calculus of variations
but also in solving fractional difference equations
containing left and right fractional differences.
One of the subjects that deserves special attention
is the question of existence of solutions to the discrete
fractional Euler--Lagrange equations. Note that the obtained
fractional equation \eqref{EL} involves both the left and the
right discrete fractional derivatives. Other interesting directions
of research consist to study optimality conditions for more general variable
endpoint variational problems \cite{Zeidan,AD:10b,MyID:169};
isoperimetric problems \cite{Almeida1,AlmeidaNabla};
higher-order problems of the calculus of variations
\cite{B:J:05,RD,NataliaHigherOrderNabla};
to obtain fractional sufficient optimality
conditions of Jacobi type and a version of Noether's theorem
\cite{Bartos,Cresson:Frederico:Torres,gastao:delfim,MyID:149}
for discrete-time fractional variational problems;
direct methods of optimization for absolute extrema
\cite{Bohner:F:T,mal:tor,T:L:08};
to generalize our fractional first and second order
optimality conditions for a fractional Lagrangian
possessing delay terms \cite{B:M:J:08,M:J:B:09};
and to generalize the results from $(h\mathbb{Z})_a$
to an arbitrary time scale $\mathbb{T}$.


\section*{Acknowledgments}

This work is part of the first author's PhD project carried
out at the University of Aveiro under the framework of the Doctoral Programme
\emph{Mathematics and Applications} of Universities of Aveiro and Minho.
The financial support of the Polytechnic Institute of Viseu and
\emph{The Portuguese Foundation for Science and Technology} (FCT),
through the ``Programa de apoio \`{a} forma\c{c}\~{a}o avan\c{c}ada
de docentes do Ensino Superior Polit\'{e}cnico'',
PhD fellowship SFRH/PROTEC/49730/2009, is here gratefully acknowledged.
The second author was supported by FCT through the PhD fellowship
SFRH/BD/39816/2007; the third author by FCT through the R\&D
unit \emph{Centre for Research on Optimization and Control} (CEOC)
and the project UTAustin/MAT/0057/2008.



\end{document}